\documentclass[hidelinks,onefignum,onetabnum]{siamart220329}

\usepackage{lipsum}
\usepackage{amsfonts,amssymb,amsmath}
\usepackage{mathtools}
\usepackage{nicefrac}
\usepackage{enumitem}
\usepackage{graphicx}
\usepackage{epstopdf}
\usepackage[algo2e]{algorithm2e}
\usepackage{tikz}
\usepackage{pgfplots}
\pgfplotsset{compat=1.18} 
\usetikzlibrary{calc,plotmarks}
\usepackage[outputdir=./,frozencache]{minted}
\usepackage{cprotect}
\definecolor{bg}{rgb}{0.95,0.95,0.95}
\ifpdf
  \DeclareGraphicsExtensions{.eps,.pdf,.png,.jpg}
\else
  \DeclareGraphicsExtensions{.eps}
\fi

\newcommand\myR{\mathbb{R}_{>}}

\newsiamremark{example}{Example}
\newsiamremark{remark}{Remark}
\newsiamremark{hypothesis}{Hypothesis}
\crefname{hypothesis}{Hypothesis}{Hypotheses}
\newsiamthm{claim}{Claim}

\headers{Stochastic $p$th root approximation of a stochastic matrix}{F. Durastante, B. Meini}

\title{Stochastic $p$th root approximation of a stochastic matrix: A Riemannian optimization approach\thanks{Submitted to the editors \today.
\funding{This work has been partially supported by:  Spoke 1 ``FutureHPC \& BigData''  of the%
Italian Research Center on High-Performance Computing, Big Data and Quantum%
Computing (ICSC)  funded by MUR Missione 4 Componente 2 Investimento 1.4:%
Potenziamento strutture di ricerca e creazione di ``campioni nazionali di R\&S%
(M4C2-19 )'' - Next Generation EU (NGEU); by the 
``INdAM – GNCS Project: Metodi basati su matrici e tensori strutturati per problemi di algebra lineare di grandi dimensioni'' code CUP\_E53C22001930001; and by the PRIN project ``Low-rank Structures and Numerical Methods in Matrix and Tensor Computations and their Application'' code 20227PCCKZ. The authors are member of the INdAM GNCS group.
}}}

\author{Fabio Durastante\thanks{Dipartimento di Matematica, Universit\`{a} di Pisa, Pisa (PI), Italy
  (\email{fabio.durastante@unipi.it}, \url{https://fdurastante.github.io/}).}
\and Beatrice Meini\thanks{Dipartimento di Matematica, Universit\`{a} di Pisa, Pisa (PI), Italy
  (\email{beatrice.meini@unipi.it}).}
}

\usepackage{amsopn}

\usepackage{booktabs}
\makeatletter
\newenvironment{tabminted}{%
	\let\FV@ListVSpace\relax  
	\minted
}{%
	\endminted
	\unskip   
	\aftergroup\@tabmintedend
}
\newcommand*{\tabminted@finalstrut}[1]{%
	\ifdim\prevdepth>0pt
	\ifdim\dp#1>\prevdepth
	\vskip\dimexpr(\dp#1)-\prevdepth\relax
	\fi
	\else
	\vskip\dimexpr(\dp#1)\relax
	\fi
}
\newcommand*{\@tabmintedend}{%
	\let\@finalstrut\tabminted@finalstrut
}
\makeatother
\usepackage[caption=false]{subfig}
\usepackage{microtype}

\ifpdf
\hypersetup{
  pdftitle={Stochastic $p$th root approximation of a stochastic matrix},
  pdfauthor={F. Durastante, B. Meini}
}
\fi

\begin{document}

\maketitle

\begin{abstract}
We propose two approaches, based on Riemannian optimization, for computing
a stochastic approximation of the $p$th root of a stochastic matrix $A$. In the first approach, the approximation is found in the Riemannian manifold of
positive stochastic matrices.
In the second approach, we introduce the Riemannian manifold of
positive stochastic matrices sharing with $A$ the Perron eigenvector and we compute the approximation of the $p$th root of $A$ in such a manifold.
This way, differently from the available methods based on constrained optimization, $A$ and its $p$th root approximation share the Perron eigenvector.
Such a property is relevant, from a modeling point of view,  in the embedding problem for Markov chains.
The extended numerical experimentation shows that, in the first approach,  the Riemannian optimization methods are generally faster and more accurate than the available methods based on constrained optimization.
In the second approach, even though the stochastic approximation of the $p$th root is found in a smaller set, the approximation is generally more accurate than the one obtained by standard constrained optimization.
\end{abstract}



\section{Introduction}

Discrete and continuous-time Markov chains are used to model a range of different time-evolving phenomena such as queueing models \cite{blm:book}, physician's estimate of prognosis under alternative treatment plans~\cite{application1}, synthetic DNA~\cite{Ardiyansyah2021}, rating agencies predicting the evolution of a firm's rating in a given time interval~\cite{application2,Hughes2016}, studying diffusion and consensus on directed graphs~\cite{Veerman2019} or the analysis of daily rainfall occurrence~\cite{application3}. 
The evolution of a discrete-time Markov chain, with  a finite number $n$ of states, is described in terms of an $n\times n$ matrix $A$, called \emph{transition matrix}, whose $(i,j)$-th entry represents the probability to go from state $i$ to state $j$ in one unit of time. The matrix $A$ is \emph{stochastic}, i.e., belongs to the set
\[
\mathbb{S}_n^0 = \{ S \in \mathbb{R}^{n \times n} \,:\; S \mathbf{1} = \mathbf{1}, \; S \geq 0 \},
\]
where $\mathbf{1} = (1,1,\ldots,1)^T \in \mathbb{R}^n$, and the symbol ``$\geq$'' represents the element-wise ordering{; see~\cite[Chapter~8]{bp:book} for an introduction to finite Markov chains}. 
For the Perron-Frobenius theorem, any stochastic matrix $A$ has a nonnegative Perron eigenvector, i.e., a nonnegative vector $\boldsymbol{\pi}\ne 0$ such that $\boldsymbol{\pi}^T A=\boldsymbol{\pi}^T$; when $\boldsymbol{\pi}$ is normalized so that $\boldsymbol{\pi}^T\mathbf{1}=1$, then $\boldsymbol{\pi}$ is called steady state vector, or stationary distribution, for the matrix $A$. If $A$ is irreducible, then the steady state vector  has positive entries and is unique, moreover $\lim_{k\to\infty}A^k=\mathbf{1} \boldsymbol{\pi}^T$.

In many applications, the entries of the matrix $A$ are estimated through the analysis of historical series over long time intervals. 
Therefore, the unit time at which transitions occur is generally larger, compared with the characteristic time of the phenomenon to be analyzed. To know the transition probabilities in the typical time step of the phenomenon, it would therefore be necessary to investigate what happens in a fraction of a unit of time: for instance, which are the transition probabilities in a half-time unit?
An attempt might be computing a matrix $X$ such that
$A  = X^2$,
or, in other terms, a \emph{square root} of the transition matrix $A$. More generally, we can inquire about any number of intermediate steps $p$ thus looking for a $p$th root $X$ of $A$,
$
A = X^p$, $p \in \mathbb{N}$.
However, for the matrix $X$ to be descriptive of a Markov process, we need it to be itself a transition matrix, that is, $X$ should satisfy 
\[
X^p = A, \; X \in\mathbb{S}_n^0.
\]
Unfortunately, such $X$ does not exist in general~\cite{HighamLijing2010}, and several pathological cases can be readily produced, e.g., the $p$th root may exist or not, it may exist and not be stochastic, and there can even be more than one stochastic $p$th root. In some cases, we can exploit the fact that $f(z)= z^{\nicefrac{1}{p}}$ has more than one branch in the complex plane to define \emph{non-primary} matrix functions $A^{\nicefrac{1}{p}}$ by selecting different determinations of the function on repeated eigenvalues --- see~\cite[Section~1.4]{HighamBook} --- and look for a non-primary stochastic matrix $p$th root. Even with this added degree of freedom, stochastic $p$th roots of a stochastic matrix might not exist. 

The problem of the existence of a stochastic $p$th root of a stochastic matrix is also strictly related to the so called embedding problem for Markov chains (see \cite{vanbrunt}). Indeed, a Markov chain
with transition matrix $A$ is \emph{embeddable} if and only if there exists a rate matrix $Q$ such
that $A=\exp(Q)$. We recall that a rate matrix is a matrix with nonnegative off-diagonal entries, such that $Q\mathbf{1}=\mathbf{0}$. It is immediate to verify that, if a Markov chain is embeddable, then $X=\exp(Q/p)$ is a stochastic $p$th root of $A$, for any $p$. More precisely, in \cite{Kingman} it is shown that a Markov chain is embeddable if and only if the transition matrix $A$ is nonsingular and has stochastic $p$th roots of any order~$p$. To this regard, in \cite{vanbrunt} a characterization of embeddable Markov chains is given in terms of infinite divisibility properties of nonnegative matrices. In practice, these conditions are difficult to {verify} and sufficient conditions for embeddability have been introduced for specific cases, as  equal-input, circulant, symmetric or doubly stochastic matrices~\cite{Baake2020}, small size matrices~\cite{Casanellas2020}, or in other frameworks \cite{Davies2010,Ekhosuehi2023,HighamLijing2010,Bhat2020}.

In the case where a stochastic $p$th root does not exist,  an alternative approach consists {of} finding an approximation which is a stochastic matrix. 
To this end, there are some methods available in the literature relying on optimization strategies, see, e.g., the code package in~\cite{Pfeuffer2017}. 
Given $A\in\mathbb{S}_n^0$, the main approach consists in the computation of the solution $X$  of the following constrained optimization problem, where $\|\cdot\|_F$ is the Frobenius norm (see \cite{HighamLijing2010,https://doi.org/10.1002/sim.2970}):
\begin{enumerate}[label=(\emph{\alph*}),ref=(\emph{\alph*})]
	\item\label{alg:a2} find
	\[
	X = \arg\min_{ X  \in \mathbb{S}_n^0 } \frac{1}{2}\| X^p - A\|_F^2.
	\]
\end{enumerate}
Other, less used, strategies consist in
\begin{enumerate}[label=(\emph{\alph*}),ref=(\emph{\alph*})]
	\setcounter{enumi}{1}
	\item\label{alg:a1} find
	\[
	X = \arg\min_{ X  \in \mathbb{S}_n^0 } \frac{1}{2}\| X - A^{\nicefrac{1}{p}} \|_F^2,
	\]
	where $A^{\nicefrac{1}{p}}$ is the principal $p$th root of $A$;
	\item\label{alg:a3} find 
	\[
	\mathbf{h} = \arg\min_{ \mathbf{h} \in \Omega  } \left\| \left( \sum_{i=0}^{n-1} h_i A^i \right)^p - A \right\|_F^2,
	\]
	for
	\[
	\Omega = \{ \mathbf{h} \in \mathbb{R}^n \,:\,\mathbf{1}^T \mathbf{h} = 1,\; B \mathbf{h} \geq 0,  \; B = [\operatorname{vec}(I)|\operatorname{vec}(A)|\cdots|\operatorname{vec}(A^{n-1})] \},
	\]
	where $\operatorname{vec}(F)$ is the vector obtained by stacking the columns of the matrix $F$,
	and set $X = X(\mathbf{h}) = \sum_{i=0}^{n-1} h_i A^i$.
\end{enumerate}

Formulations \ref{alg:a2} and \ref{alg:a1} deliver an approximation $X$ that is not, in general, a matrix-function of $A$, while in \ref{alg:a3} the approximation is a primary matrix function by construction. 
However, as pointed out in \cite{HighamLijing2010}, there are situations where the stochastic $p$th root exists but it is not a matrix-function of $A$, therefore \ref{alg:a3} does not compute such a stochastic $p$th root. On the other hand,  since $X$ is a function of $A$, a nice feature of formulation \ref{alg:a3} is that the output matrix $X$ shares with $A$ the steady state vector.
If both $A$ and $X$ are irreducible, this implies that $\lim_{k\to\infty}X^k=\lim_{k\to\infty}A^k=\mathbf{1}\boldsymbol{\pi}^T$, i.e., the asymptotic behavior of the Markov chains with transition matrices $X$ and $A$, respectively, is the same. From the modeling point of view, this is a desirable property, since we expect that doing time steps of different ``lengths'' should always bring us to the same limit.

In this paper, we propose to compute a stochastic approximation $X$ to a $p$th root of the stochastic matrix $A$ by relying on a Riemannian optimization approach. Indeed, it is well known that the set of positive stochastic matrices is a Riemannian manifold, called multinomial manifold \cite{Douik8861409}. Therefore, the first approach that we propose is to solve problem \ref{alg:a2} in this Riemannian optimization setting. 
However, as in standard constrained optimization, the computed matrix $X$ is a stochastic matrix that generally does not share with $A$ the steady state vector. 
To overcome this drawback, given a positive vector  $\boldsymbol{\pi} \in \mathbb{R}^{n}$ such that $\boldsymbol{\pi}^T \mathbf{1} = 1$, 
we introduce the Riemannian manifold $\mathbb{S}_n^{\boldsymbol{\pi}}$ of positive stochastic matrices, having $\boldsymbol{\pi}$ as steady state vector. Such Riemannian manifold can be seen as the generalization of the Riemannian manifold of doubly stochastic matrices,
which corresponds to the special case where $\boldsymbol{\pi}=\frac1n\mathbf{1}$. In order to apply the Riemannian optimization algorithms, we give an expression to the tangent space, to the orthogonal complement and orthogonal projection, and to the Riemannian gradient and Hessian, by extending the analog properties valid for doubly stochastic matrices. To define {a} retraction from the tangent bundle to the manifold, we use a generalization of the Sinkhorn-Knopp algorithm. Hence, given an irreducible stochastic matrix $A$ with stationary distribution $\boldsymbol{\pi}$, we approximate its stochastic $p$th root by solving~\ref{alg:a2} in the manifold $\mathbb{S}_n^{\boldsymbol{\pi}}$. In implementing the optimization algorithms we need to solve several singular symmetric linear systems; in this regard, we provide a lower and an upper bound to the nonzero eigenvalues of the matrix, which give information on the convergence of Conjugate Gradient-like methods. Moreover, we propose some {preconditioners} to improve the convergence of iterative methods for the solution of such linear systems. 

The new Riemannian manifold $\mathbb{S}_n^{\boldsymbol{\pi}}$ has been implemented in Matlab, in a format compatible with the Manopt library \cite{manopt}. The code is available in the GIT repository \href{https://github.com/Cirdans-Home/pth-root-stochastic}{github.com/Cirdans-Home/pth-root-stochastic}. 

The proposed methods have been tested on a variety of stochastic matrices $A$, with different properties, in terms of size and embeddability. An application to finance, where $A$ represents the transitions in the credit classes \cite{application2,Hughes2016}, has been treated in detail.

In the cases where we are not interested in preserving the steady state vector, the comparisons between constrained optimization algorithms,  integrating together trust region and interior point techniques~\cite{interiorpoint}, and the Riemannian-based optimizers for formulation~\ref{alg:a2} on the multinomial manifold, show that the latter achieve smaller or equal residuals, and they are generally faster. 
When we are interested in preserving the steady state vector, the numerical experiments  on the Riemannian manifold $\mathbb{S}_n^{\boldsymbol{\pi}}$ show that, in some cases, the approximation of the $p$th root of $A$ has a higher residual with respect to the approximation in the manifold $\mathbb{S}_n$ of positive stochastic matrices. This is expected, since the set $\mathbb{S}_n^{\boldsymbol{\pi}}$ is smaller than the set $\mathbb{S}_n$. However, in general, optimization methods in set $\mathbb{S}_n$ provide an approximation having a stationary distribution far from the stationary distribution $\boldsymbol{\pi}$ of $A$.
In the application to credit ranking in finance, the matrix $A$ is reducible, therefore we apply our method to the irreducible stochastic matrix $\widetilde A=\gamma A+(1-\gamma)\frac1n \boldsymbol{1}\boldsymbol{1}^T$, $0<\gamma<1$, which resembles the Page Rank matrix. From the numerical experiments, the approximation of the $p$th root has a structure close to the reducible structure of $A$, and the numerical values of its entries are very realistic, from a modeling point of view.

The paper is organized as follows. In Section~\ref{sec:riem_opt_general} we recall the main definitions concerning Riemannian manifolds and Riemannian optimization. 
In Section~\ref{sec:stochastic-old-methods} we recall the properties of the multinomial manifold of positive stochastic matrices and solve problem~\ref{alg:a2} in the framework of Riemannian optimization in such a manifold. In Section~\ref{sec:stochastic-new-methods} we introduce the Riemannian manifold $\mathbb{S}_n^{\boldsymbol{\pi}}$ of stochastic matrices having a common steady state vector $\boldsymbol{\pi}$, and solve problem~\ref{alg:a2} in this manifold. We present the numerical experiments in Section~\ref{sec:numer_ex} and draw conclusions in Section~\ref{sec:conclusions}.

\subsection{Notation}

In the following, the symbols ``$\oslash$'' and ``$\odot$'' represent the Hadamard (entry-wise) matrix division and multiplication, respectively. Given a vector $\mathbf{x}$ -- always denoted in bold face -- the symbol ``$\operatorname{diag}(\mathbf{x})$'' denotes the diagonal matrix having the entries of $\mathbf{x}$ on the main diagonal; for notational simplicity, we will also denote $D_{\mathbf{x}} = \operatorname{diag}(\mathbf{x})$. If $A$ is a square matrix, then ``$\operatorname{diag}(A)$'' denotes the vector formed by the diagonal entries of~$A$, and $\lambda(A)$ its spectrum. The notation concerning Riemannian geometry will be introduced at the time of their~use. If $A\in\mathbb{R}^{m\times n}$, $A$ is said to be nonnegative (positive), and we write $A\ge 0$ ($A>0$), if all its entries are nonnegative (positive).

\section{Preliminaries on Riemannian optimization}\label{sec:riem_opt_general}
We start by recalling some definitions concerning Riemannian manifolds and Riemannian optimization. The interested reader may find more details on this subject in~\cite{Douik8861409,AbsilBook,MR4533407}.

{\begin{definition}[Embedded Manifold]\label{ref:manifold-definition}
Let $\mathcal{E}$ be a linear space of dimension $d$. A non-empty subset $\mathcal{M}$ of $\mathcal{E}$ is a smooth embedded submanifold of $\mathcal{E}$ of dimension $n$ if either
\begin{itemize}
    \item $n = d$ and $\mathcal{M}$ is open in $\mathcal{E}$;
    \item $n = d-k$ for some $k \geq 1$ and, for each $x \in \mathcal{M}$, there exists a neighborhood $U$ of $x$ in $\mathcal{E}$ and a smooth function $h : U \rightarrow \mathbb{R}^k$ such that
    \begin{itemize}
        \item If $y \in U$, then $h(y)= 0$ if and only if $y \in \mathcal{M}$; and
        \item $\operatorname{rank} \mathrm{D}h(x) = k$, for $\mathrm{D}h(x)$ the differential of $h$ at $x$;
    \end{itemize}
    such function $h$ is called a \emph{local defining function} for $\mathcal{M}$ at $x$.
\end{itemize}
\end{definition}}

A tangent vector to $\mathcal{M}$ at a point $x$ is defined as follows:

\begin{definition}[Tangent vector, tangent bundle]\label{def:tangent_vector}
	A tangent vector $\xi_x$ to a manifold $\mathcal{M}$ at a point $x$ is a mapping from the set  $\mathfrak{F}_x(\mathcal{M})$ of smooth real-valued functions defined on a neighborhood of $x$ to $\mathbb{R}$ such that there exists a curve $\gamma$ on $\mathcal{M}$ realizing the tangent vector $\xi_x$, i.e., such that $\gamma(0)=x$, and
	\[
	\xi_x f = \dot{\gamma}(0)f \triangleq \left.\frac{\mathrm{d}(f(\gamma(t)))}{\mathrm{d}t}\right\rvert_{t=0}, \; \forall \, f \in  \mathfrak{F}_x(\mathcal{M});
	\]
	see, e.g., Figure~\ref{fig:tangent_space}. The \emph{tangent space} $\mathcal{T}_x \mathcal{M}$ at $x\in\mathcal M$ is then the set of all tangent vectors to $\mathcal{M}$ at a point $x$. {The \emph{tangent bundle} is} the manifold $\mathcal{T} \mathcal{M}$ that assembles all the tangent vectors, i.e., {the disjoint union} $\mathcal{T} \mathcal{M} = {\bigsqcup_{x \in \mathcal{M}}} \mathcal{T}_x \mathcal{M}$.
\end{definition}

\begin{figure}[htbp]
	\centering
	\definecolor{c0000b0}{RGB}{0,0,176}
	\definecolor{cff5500}{RGB}{255,85,0}
	\definecolor{cff0000}{RGB}{255,0,0}
	\definecolor{cfe0000}{RGB}{254,0,0}
	\definecolor{c0000e3}{RGB}{0,0,227}
	\definecolor{c0000ff}{RGB}{0,0,255}
	\definecolor{c0000cc}{RGB}{0,0,204}
	\def \globalscale {0.500000}
	\begin{tikzpicture}[y=0.80pt, x=0.80pt, yscale=-\globalscale, xscale=\globalscale, inner sep=0pt, outer sep=0pt]
		\begin{scope}[cm={{1.33333,0.0,0.0,-1.33333,(0.0,317.333)}}]
			\path[draw=black,fill=c0000b0,fill opacity=0.100,line cap=butt,line
			join=round,line width=0.960pt,miter limit=10.00] (97.8672,237.3980) --
			(0.6020,155.5270) -- (259.5160,23.2730) -- (346.2850,122.6410) -- cycle;
			\path[cm={{0.75,0.0,0.0,-0.75,(0.0,238.0)}},draw=black,fill=cff5500,fill
			opacity=0.100,line cap=butt,line join=round,line width=1.280pt,miter
			limit=10.00] (254.9824,80.8730) .. controls (223.3707,81.1902) and
			(195.5078,97.4730) .. (166.0996,118.8438) .. controls (152.5009,128.8851) and
			(120.0000,160.9902) .. (120.0000,160.9902) -- (45.8164,248.4590) .. controls
			(153.5850,362.1742) and (377.6870,280.3265) .. (453.5156,268.0938) .. controls
			(401.6404,235.3896) and (379.6382,88.0933) .. (277.5312,83.3379) .. controls
			(269.7720,81.5768) and (262.2774,80.7999) .. (254.9824,80.8730) -- cycle;
			\path[draw=cff0000,dash pattern=on 0.80pt off 2.40pt,line cap=butt,line
			join=round,line width=0.960pt,miter limit=10.00] (69.9840,49.1449) .. controls
			(124.6410,105.8670) and (141.8240,120.9840) .. (155.5940,129.6450) .. controls
			(169.3590,138.3050) and (179.7110,140.5120) .. (189.5120,140.0980) .. controls
			(199.3090,139.6840) and (208.5590,136.6480) .. (219.5310,129.6760) .. controls
			(230.5040,122.7070) and (243.1990,111.8050) .. (257.5200,95.7930) .. controls
			(271.8400,79.7852) and (287.7810,58.6641) .. (287.7810,58.6641);
			\path[fill=black,even odd rule,line width=0.080pt] (69.9840,49.1449) --
			(76.5199,52.5660) -- (73.1641,55.8012);
			\path[draw=cff0000,line cap=butt,line join=round,line width=0.960pt,miter
			limit=10.00] (69.9840,49.1449) -- (76.5199,52.5660) -- (73.1641,55.8012) --
			cycle;
			\path[draw=cff0000,dash pattern=on 0.80pt off 2.40pt,line cap=butt,line
			join=round,line width=0.960pt,miter limit=10.00] (146.5860,159.6950) ..
			controls (171.0160,156.8010) and (179.2970,148.5160) .. (184.2660,141.8910) ..
			controls (189.2340,135.2660) and (190.8910,130.2970) .. (190.7890,125.1250) ..
			controls (190.6840,119.9490) and (188.8200,114.5660) .. (178.8830,108.7700);
			\path[fill=black,even odd rule,line width=0.080pt] (146.5860,159.6950) --
			(153.2620,156.5590) -- (153.8130,161.1880);
			\path[draw=cfe0000,line cap=butt,line join=round,line width=0.960pt,miter
			limit=10.00] (146.5860,159.6950) -- (153.2620,156.5590) -- (153.8130,161.1880)
			-- cycle;
			\path[fill=cff0000,even odd rule,line width=0.320pt,miter limit=10.00]
			(187.1630,140.2329) .. controls (187.1630,144.0999) and (181.3680,144.0999) ..
			(181.3680,140.2329) .. controls (181.3680,136.3738) and (187.1630,136.3738) ..
			(187.1630,140.2329);
			\path[draw=c0000e3,line cap=butt,line join=round,line width=0.960pt,miter
			limit=10.00] (183.9220,141.1450) -- (112.7030,121.6840);
			\path[fill=black,even odd rule,line width=0.080pt] (112.7030,121.6840) --
			(120.0700,121.2810) -- (118.8440,125.7770);
			\path[draw=c0000e3,line cap=butt,line join=round,line width=0.960pt,miter
			limit=10.00] (112.7030,121.6840) -- (120.0700,121.2810) -- (118.8440,125.7770)
			-- cycle;
			\path[draw=c0000ff,fill=c0000cc,fill opacity=0.100,draw opacity=0.743,line
			cap=butt,line join=round,line width=0.960pt,miter limit=10.00]
			(184.3360,141.5590) -- (155.7700,186.6910);
			\path[fill=black,even odd rule,line width=0.080pt] (155.7700,186.6910) --
			(157.5430,179.5310) -- (161.4800,182.0230);
			\path[draw=c0000ff,fill=c0000cc,fill opacity=0.100,draw opacity=0.743,line
			cap=butt,line join=round,line width=0.960pt,miter limit=10.00]
			(155.7700,186.6910) -- (157.5430,179.5310) -- (161.4800,182.0230) -- cycle;
			\path[xscale=1.000,yscale=-1.000,line width=0.600pt] (64.0917,-38.1447)
			node[above right] (text20223) {$x_0(t)$};
			\path[xscale=1.000,yscale=-1.000,line width=0.600pt] (167.9969,-99.1264)
			node[above right] (text23109) {$x_1(t)$};
			\path[xscale=1.000,yscale=-1.000,line width=0.600pt] (155.7700,-186.6910)
			node[above right] (text26097) {$\xi_x^1$};
			\path[xscale=1.000,yscale=-1.000,line width=0.600pt] (91.8765,-129.0046)
			node[above right] (text33699) {$\xi_x^2$};
			\path[xscale=1.000,yscale=-1.000,line width=0.600pt] (186.5802,-140.5539)
			node[above right] (text36203) {$A$};
		\end{scope}
	\end{tikzpicture}
	
	\caption{Tangent space (opaque blue) of a bi-dimensional manifold embedded (red) in $\mathbb{R}^3$. The tangent space $\mathcal{T}_x\mathcal{M}$ is computed by taking derivatives of the curves $x_0(t)$ and $x_1(t)$ (red dotted lines) going through ${A}$ at the origin.}
	\label{fig:tangent_space}
\end{figure}
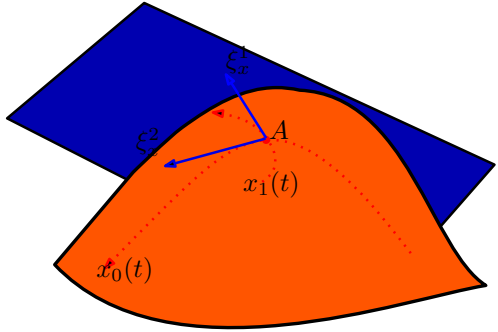
{To further characterize the tangent at a point in the case of an embedded manifold, such as the ones we are interested in, we also report the following result.}
{\begin{theorem}[{\cite[Theorem 3.15]{MR4533407}}]\label{thm:tangent_characterization}
    Let $\mathcal{M}$ be an embedded submanifold of $\mathcal{E}$. Consider $x \in \mathcal{M}$ and the set $\mathcal{T}_x \mathcal{M}$ from Definition~\ref{def:tangent_vector}. If $\mathcal{M}$ is an open submanifold, then $\mathcal{T}_x \mathcal{M}=\mathcal{E}$. Otherwise $\mathcal{T}_x \mathcal{M} = \ker \mathrm{D}h(x)$ with $h$ any local defining function at $x$.
\end{theorem}}

A Riemannian manifold $\mathcal{M}$ is a manifold equipped with a positive-definite inner product on its tangent space, i.e., {$\langle \xi_x, \eta_x\rangle_x$, for any $\xi_x,\eta_x\in\mathcal{T}_x\mathcal{M}$.
Such a \emph{metric}, called Riemannian metric, induces the norm $\| \xi_x \|_x=\sqrt{\langle \xi_x, \xi_x\rangle_x}$, for any $\xi_x\in\mathcal{T}_x\mathcal{M}$.}

On the Riemannian manifold, we can define the 
minimization problem
{\begin{equation}\label{eq:genopt}
	\displaystyle \arg\min_{ x  \in \mathcal{M}} f(x),
\end{equation}
}
where $f:\mathcal{M}\to \mathbb{R}$ is a suitable smooth function. 
\subsection{Optimization methods}\label{sec:optstoc}
Optimization methods for solving \eqref{eq:genopt} on Riemannian manifolds use local pull-back from the tangent spaces {$\mathcal{T}_x \mathcal{M}$} to the manifold $\mathcal{M}$ to produce a sequence of iterates, which can be interpreted as iterates moving along specific curves on the manifold, see the representation in Figure~\ref{fig:basic_idea}.
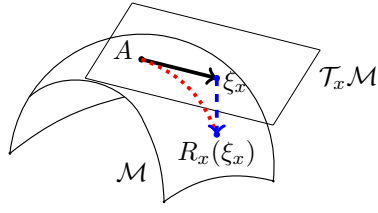
\begin{figure}[htbp]
	\centering
	\begin{tikzpicture}[scale=0.25]
		\filldraw (0,0) circle (1.35pt);
		\filldraw (14.01,-1.485) circle (1.05pt);
		\filldraw (8.195,-2.585) circle (1.15pt);
		\draw [line width=0.15mm, black ] (0,0) to [bend left=45] (7.8,6.3) to [bend left=45] (14,-1.5);
		\draw [line width=0.15mm, black ] (1.035,3) to [bend left=55] (7.1,1.7) to [bend left=12] (8.2,-2.6);
		\draw [line width=0.15mm, black ] (8.2,-2.57) to [bend left=25] (14,-1.47);
		\draw [line width=0.15mm, black ] (0,0) to [bend left=12] (6.1,3);
		\node at (6.5,-1) {$\mathcal{M}$};
		\filldraw (7,5) circle (4pt);				
		\node at (6,5.5) {$A$};
		\draw [line width=0.10mm] (4,4) -- ++(10,-2.5);
		\draw [line width=0.10mm] (4,4) -- ++(2.5,4);
		\draw [line width=0.10mm] (6.5,8) -- ++(10,-2.5);
		\draw [line width=0.10mm] (14,1.5) -- ++(2.5,4);
		\node at (18,4) {$\mathcal{T}_x \mathcal{M}$};
		\draw [->,line width=0.5mm](7,5) -- ++(4,-1);	
		\filldraw[blue] (11,4) circle (4pt);
		\node at (12,3.8) {$\xi_x$};
		\draw [dotted,line width=0.5mm,red] (7,5) to [bend left=25] node{.}(11,1);
		\filldraw[blue] (11,1) circle (4pt);
		\node at (11,0) {$R_{x}(\xi_x)$};
        \draw[dashed,blue,->,line width=0.5mm] (11,4) to (11,1);
	\end{tikzpicture}
	\caption{The basic idea of Riemannian optimization algorithms: evolving the iterates using local pull-back (retraction) from the tangent space to the manifold.}
	\label{fig:basic_idea}
\end{figure}
What distinguishes the different algorithms is how the new point on the tangent space is determined. In general, it is possible to adapt the different classes of optimization methods in this new context, consider, e.g., first order methods, Newton and Quasi-Newton methods, or Trust-Region methods, see, \cite[Chapters~6,7 and 8]{AbsilBook} for a complete discussion. 

In order to define these methods, we need to recall some differential structures for functions taking values on the manifold. 
We denote by {$\mathrm{D} f(x)[\xi]$}  the directional derivative of $f$ given by:
{\[
\mathrm{D} f(x)[\xi]=\lim_{t\to 0}\frac{f(x+t\xi)-f(x)}{t}.
\]
}

\begin{definition}[Affine connection]\label{def:affine-connection}
	An \emph{affine connection} $\nabla \,:\,\mathcal{T}\mathcal{M} \times \mathcal{T}\mathcal{M} \to \mathcal{T}\mathcal{M}$ is a map that associates to each $(\eta,\xi)$ in the tangent bundle (Definition~\ref{def:tangent_vector}) the tangent vector $\nabla_\eta \xi$ satisfying for all  $a,b \in \mathbb{R}$, and smooth $f,g: \mathcal{M} \longrightarrow \mathbb{R}$:
	\begin{itemize}
		\item $\nabla_{f(\eta)+g(\chi)}\xi =  f(\nabla_\eta \xi)+ g(\nabla_\chi \xi)$
		\item $\nabla_{\eta}(a\xi+b\varphi) = a\nabla_{\eta}\xi+b\nabla_{\eta}\varphi$
		\item $\nabla_{\eta}(f(\xi)) = \xi(f) \eta + f(\nabla_\eta \xi)$,
	\end{itemize}
	wherein the vector field $\xi$ acts on the function $f$ by derivation, that is 
	\[\xi(f)=\mathrm{D}(f)[\xi].\]
	We call \emph{Levi-Civita connection} the affine connection that preserves the Riemannian metric, i.e., the affine connection such that
	\begin{itemize}
		\item $\nabla_\eta \xi - \nabla_\xi \eta = [\eta,\xi] $ $\forall\,\eta,\xi \in \mathcal{T}\mathcal{M}$,
		\item $\chi \langle \eta,\xi \rangle = \langle \nabla_\chi \eta,\xi \rangle + \langle\eta ,  \nabla_\chi \xi \rangle$, $\forall\,\eta ,\xi ,\chi \in \mathcal{T} \mathcal{M}$,
	\end{itemize}
	where we are denoting with $[\cdot,\cdot]$ the Lie bracket
	\[
	[\xi,\eta]g = \xi(\eta (g)) - \eta(\xi (g)).
	\]
\end{definition}

\begin{definition}[Riemannian Gradient and Hessian]\label{def:riemannian-gradient-and-hessian}
	The \emph{Riemannian gradient} of $f$ at {$x$}, denoted by {$\operatorname{grad}f(x)$}, of a manifold $\mathcal{M}$ is defined as the unique vector in {$\mathcal{T}_x\mathcal{M}$} that satisfies:
	{\begin{align*}
		\langle \operatorname{grad}f(x), \xi_x \rangle_x = \mathrm{D} f(x) [\xi_x],\ \forall \ \xi_x \in \mathcal{T}_x\mathcal{M}.
	\end{align*}}
	The \emph{Riemannian Hessian} of $f$ at {$x$}, denoted by {$\operatorname{{hess}}f(x)$}, of a manifold $\mathcal{M}$ is a mapping from {$\mathcal{T}_x\mathcal{M}$} into itself defined by:
	{\begin{align*}
		\operatorname{{hess}}f(x)[\xi_x] = \nabla_{\xi_x} \operatorname{grad}f(x), \ \forall \ \xi_x \in \mathcal{T}_x\mathcal{M},
	\end{align*}}
	where {$\operatorname{grad}f(x)$} is the Riemannian gradient and $\nabla$ is the {Levi-Civita connection} on~$\mathcal{M}$.
\end{definition}

\begin{definition}[Retraction]\label{def:retraction}
	A retraction on a manifold $\mathcal{M}$ is a smooth map $R$ from the tangent bundle {$\mathcal{T} \mathcal{M} = \bigsqcup_{x \in \mathcal{M}} \mathcal{T}_x\mathcal{M}$, i.e., the disjoint union of the tangent spaces,} onto $\mathcal{M}$. For all $x \in \mathcal{M}$ the restriction of $R$ to $\mathcal{T}_x \mathcal{M}$, denoted by $R_x$, satisfies the following properties:
	\begin{itemize}
		\item $R_x(0) = x$ (centering);
		\item The curve $\gamma_{\xi_x}(\tau) = R_x(\tau \xi_x)$ satisfies
		\[
		\left.\frac{d \gamma_{\xi_x}(\tau)}{d\tau}\right|_{\tau = 0} = \xi_x, \quad\,\forall\, \xi_x \in \mathcal{T}_x \mathcal{M}. \qquad \text{(local rigidity)}
		\]
	\end{itemize}
\end{definition}

{To characterize retractions for embedded manifolds  we will use the following result, that we report here for the sake of completeness.}

\begin{theorem}[{\cite[Proposition 4.1.2]{AbsilBook}}]\label{thm:retraction-on-embedded-manifold}
	Let $\mathcal{M}$ be an embedded manifold of the Euclidean space $\mathcal{E}$ and let $\mathcal{N}$ be an abstract manifold such that $\dim(\mathcal{M}) + \dim(\mathcal{N}) = \dim(\mathcal{E})$. Assume that there is a diffeomorphism
	\begin{align*}
		\phi: \mathcal{M} \times \mathcal{N} &\longrightarrow \mathcal{E}^*  \\
		(A,B) &\longmapsto \phi(A,B)
	\end{align*}
	where $\mathcal{E}^*$ is an open subset of $\mathcal{E}$, with a neutral element $I \in \mathcal{N}$ satisfying
	\begin{align*}
		\phi(A,I) = A, \ \forall \ A \in \mathcal{M}.
	\end{align*}
	Under the above assumption, the mapping 
	\begin{align*}
		R_x: \mathcal{T}_x\mathcal{M} &\longrightarrow \mathcal{M}  \\
		\xi_x  &\longmapsto  R_x(\xi_x) = \pi_1(\phi^{-1}(x+\xi_x)),
	\end{align*}
	where $\pi_1: \mathcal{M} \times \mathcal{N} \longrightarrow \mathcal{M}: (A,B) \longmapsto A$ is the projection onto the first component, defines a retraction on the manifold $\mathcal{M}$ for all $x \in \mathcal{M}$ and $\xi_x$ in the neighborhood of $0_x$.
\end{theorem}

The general sketch of Newton's method for solving~\eqref{eq:genopt} on a Riemannian manifold is synthesized in Algorithm~\ref{alg:newton}. For other optimization methods, we refer to \cite[Chapters 4,6--8]{AbsilBook}.
\begin{algorithm2e}[htbp]
	\caption{Newton's method for~\eqref{eq:genopt} on a Riemannian Manifold $\mathcal{M}$}\label{alg:newton}
	\KwData{Manifold $\mathcal{M}$, function $f$, retraction $R$, affine connection $\nabla$, and convergence tolerance $\epsilon$}
	Initial guess $A \in \mathcal{M}$\;
	\While{$||\operatorname{grad}f(x)||_{x} \geq \epsilon$}{
		Find descent direction $\xi_x \in \mathcal{T}_x\mathcal{M}$ such that:
		\begin{align*}
			\operatorname{{hess}}f(x)[\xi_x] = -\operatorname{grad}f(x),
		\end{align*}
		wherein $\operatorname{{hess}}f(x)[\xi_x] = \nabla_{\xi_x} \operatorname{grad}f(x)$\;
		Retract $A = R_x(\xi_x)$\;
	}
	\KwResult{Output $A$.}
\end{algorithm2e}

\section{Stochastic \texorpdfstring{$p$th}{pth} root approximation via Riemannian optimization}\label{sec:stochastic-old-methods}

Here we propose an approach based on Riemannian optimization, to numerically approximate the solution of problem~\ref{alg:a2}.

Indeed, by following \cite{Douik8861409}, we endow the set of  stochastic matrices with positive entries, namely
\[
\mathbb{S}_n = \{ S \in \mathbb{R}^{n \times n} \,:\; S \mathbf{1} = \mathbf{1}, \; S > 0 \},
\]
with both a manifold structure (in the sense of Definition~\ref{ref:manifold-definition}) and an \emph{intrinsic metric}, making it a Riemannian manifold~\cite{Douik8861409,7182334}, known as multinomial manifold. The solution of problem~\ref{alg:a2} is approximated within such a manifold.

On
$\mathbb{S}_n$, we need to define 
the \emph{tangent space} $\mathcal{T}_S \mathbb{S}_n$.
By applying the definition, we find that if $S(t)$ is a smooth curve such that $S(0) = S$  and  $S(t)\in \mathbb{S}_n $ for any $t$ in a neighborhood of the origin, then the curve satisfies
\[
S(t) \mathbf{1} = \mathbf{1} \, \Rightarrow \, \dot{S}(t)\mathbf{1} = \mathbf{0},
\]
thus
$
\mathcal{T}_S \mathbb{S}_n \subseteq \{ \xi_S \in \mathbb{R}^{n \times n}\,:\; \xi_S \mathbf{1} = \mathbf{0} \},
$
while the opposite inclusion holds by comparing the number of degrees of freedom of the full space, and of the tangent space (see \cite[Proposition~1]{Douik8861409} { and Theorem~\ref{thm:tangent_characterization}}), so that
\[
\mathcal{T}_S \mathbb{S}_n = \{ \xi_S \in \mathbb{R}^{n \times n}\,:\; \xi_S \mathbf{1} = \mathbf{0} \}.
\]
Therefore, $\mathbb{S}_n$ can be extended to be a Riemannian manifold by adding a positive-definite inner product on its tangent space at every point. This metric is given by the \emph{Fisher
	information metric}
\begin{equation}\label{eq:fisher-metric}
	\begin{split}
		g_A(\xi_S,\eta_S) = 
		&\; \langle \xi_S, \eta_S \rangle_S 
		=  \sum_{i,j=1}^{n} \frac{ (\xi_S)_{i,j} (\eta_S)_{i,j} }{S_{i,j}} \\ = &\; \operatorname{Trace}( (\xi_S \oslash S) \eta_S^T ), \;\forall\, \xi_S,\eta_S \in \mathcal{T}_S \mathbb{S}_n.
	\end{split}
\end{equation}

On the multinomial manifold $\mathbb{S}_n$, we can define the analogous of problem~\ref{alg:a2} as follows:
\begin{equation}\label{eq:the_rewritten_problem}
	\text{Given }A \in \mathbb{S}_n^0 \text{ find } X = 
	\displaystyle \arg\min_{ X  \in \mathbb{S}_n } \frac{1}{2} \| X^p - A\|_F^2.
\end{equation}
The two substantial differences with respect to the standard formulation are, on the one hand, the explicit request to have the elements of $X > 0$, on the other hand, the possibility of exploiting the Riemannian manifold structure to compute the solution of problem \eqref{eq:the_rewritten_problem}. {Let us also stress that the constraint $X > 0$, needed for the definition of the metric, makes difficult obtaining general conditions of existence for the solution of~\eqref{eq:the_rewritten_problem}. In the numerical experiments, we investigated cases in which the target matrix does not satisfy this constraint.}

In particular, since the multinomial manifold is already defined in the MANOPT library~\cite{manopt},
to apply for instance the Riemannian version of the Trust Region optimization procedure, we can use a few lines of MANOPT code. 
Indeed, given a stochastic matrix \mintinline{matlab}{A} with \mintinline{matlab}{n = size(A,1)}, and  given an integer \mintinline{matlab}{p}, it is sufficient to write
\begin{minted}[bgcolor=bg,fontsize=\small]{matlab}
%
manifold = multinomialfactory(n,n); %
%
problem.M = manifold;
problem.cost = @(x) 0.5*cnormsqfro(mpower(x,p).'-A);
problem = manoptAD(problem);
options.tolgradnorm = 1e-7;
[x, xcost, info, options] = trustregions(problem,[],options);
\end{minted}
Some attention is needed since in the multinomial manifold of the MANOPT library, matrices are column stochastic instead of row stochastic. The variable \mintinline{matlab}{x}
contains an approximation to the solution of \eqref{eq:the_rewritten_problem}.

To illustrate the behavior of this approach, we consider the following example from \cite[Fact~4.10]{HighamLijing2010}.
\begin{example}\label{example:circulant-example}
	Let us consider the matrix 
	\[
	A(a) = \frac{1}{3} \begin{bmatrix}
		1-2a & 1+a & 1+a \\
		1+a & 1-2a & 1+a \\
		1+a & 1+a & 1-2a \\
	\end{bmatrix}, \quad 0 < a \leq \frac{1}{3}, 
	\]
	having eigenvalues $\{1,-a,-a\}$.
	This matrix is circulant and symmetric (and therefore doubly stochastic),  and has only {one} stochastic square root, which is not a primary function, given by
	\[
	X = \frac{1}{3} \begin{bmatrix}
		1 & 1+\sqrt{3a} & 1-\sqrt{3a} \\
		1-\sqrt{3a} & 1 & 1 + \sqrt{3a} \\
		1+\sqrt{3a} & 1-\sqrt{3a} & 1 \\
	\end{bmatrix}. %
	\]
	{This matrix is doubly stochastic and} its eigenvalues are $\{1,i\sqrt{a},-i\sqrt{a}\}$.
	Setting $a=\nicefrac{1}{6}$, an application of the optimization strategy produces
	\[
	\tilde{X} =
	\begin{bmatrix}
		0.3179  &  0.1158  &  0.5663 \\
		0.5885 &   0.3299  &  0.0816 \\
		0.0936 &   0.5543  &  0.3522 \\
	\end{bmatrix}, \; \|\tilde{X}^2 - A\|_F = 1.3102\times 10^{-12}.
	\]
	In this case the steady state vector of $\tilde X$ has 
	an absolute difference with respect to the steady state vector of $A(a)$ of $1.5701 \times 10^{-16}$, indeed this is mostly due to the fact that {there exists a stochastic square root, which is also doubly stochastic, and the optimization method converges to such a matrix.} 
\end{example}

\begin{example}
	We consider the matrix \verb|Pajek/GD96_c| from the SuiteSparse matrix collection as the adjacency matrix $A$ of an undirected graph normalized by the inverse of the sum of the row entries, then $A$ is a stochastic matrix, and we can apply the optimization strategy for $p=2$. In this case, there doesn't seem to be a stochastic square root matrix to converge to, since the residual of the optimization procedure is $\|X^2 - A\|_F = 0.52$. Furthermore, as shown in Figure~\ref{fig:eigenfailure}, the proposed approximation has a different stationary distribution with respect to $A$, therefore the ``half step'' linked to it cannot converge to the same stationary state of the global system. For this reason, in the next section we will focus on the computation of an approximation of the root that preserves the stationary distribution. 
	
	\begin{figure}[htbp]
		\centering
		\definecolor{mycolor1}{rgb}{0.00000,0.44700,0.74100}%
\definecolor{mycolor2}{rgb}{0.85000,0.32500,0.09800}%
\begin{tikzpicture}

\begin{axis}[%
width=0.85\columnwidth,
height=1.5in,
at={(0in,0in)},
scale only axis,
xmin=1,
xmax=65,
xlabel style={font=\color{white!15!black}},
xlabel={node},
ymin=0.008,
ymax=0.026,
ylabel style={font=\color{white!15!black}},
ylabel={Stationary Distribution},
axis background/.style={fill=white},
legend style={legend cell align=left, align=left, draw=none,fill=none}
]
\addplot [color=mycolor1, only marks, mark=x, mark options={solid, mycolor1}]
  table[row sep=crcr]{%
1	0.0159999999999988\\
2	0.0159999999999988\\
3	0.0160000000000017\\
4	0.0199999999999985\\
5	0.00800000000000076\\
6	0.0160000000000011\\
7	0.0160000000000015\\
8	0.0200000000000018\\
9	0.0120000000000009\\
10	0.0159999999999988\\
11	0.0159999999999987\\
12	0.0159999999999987\\
13	0.020000000000002\\
14	0.0199999999999985\\
15	0.00800000000000039\\
16	0.0160000000000016\\
17	0.0200000000000016\\
18	0.0200000000000012\\
19	0.0159999999999988\\
20	0.00800000000000081\\
21	0.0160000000000017\\
22	0.0159999999999987\\
23	0.0120000000000008\\
24	0.00800000000000079\\
25	0.0159999999999988\\
26	0.0159999999999987\\
27	0.0120000000000009\\
28	0.0199999999999985\\
29	0.0160000000000016\\
30	0.0159999999999988\\
31	0.0200000000000018\\
32	0.0119999999999993\\
33	0.0159999999999993\\
34	0.0120000000000011\\
35	0.0120000000000004\\
36	0.0199999999999985\\
37	0.0120000000000004\\
38	0.0159999999999989\\
39	0.0159999999999988\\
40	0.00800000000000058\\
41	0.0160000000000014\\
42	0.0159999999999987\\
43	0.0159999999999988\\
44	0.0199999999999985\\
45	0.0160000000000016\\
46	0.0159999999999988\\
47	0.0120000000000012\\
48	0.012\\
49	0.0159999999999989\\
50	0.0159999999999987\\
51	0.0160000000000015\\
52	0.0159999999999988\\
53	0.0120000000000004\\
54	0.0120000000000011\\
55	0.0159999999999988\\
56	0.0159999999999988\\
57	0.0159999999999988\\
58	0.0200000000000021\\
59	0.0120000000000012\\
60	0.0159999999999986\\
61	0.0240000000000023\\
62	0.0159999999999987\\
63	0.0159999999999989\\
64	0.0159999999999988\\
65	0.012000000000001\\
};
\addlegendentry{Original Matrix}

\addplot [color=mycolor2, only marks, mark=o, mark options={solid, mycolor2}]
  table[row sep=crcr]{%
1	0.0137886791156765\\
2	0.0146806590344114\\
3	0.0167106213746025\\
4	0.0170746900645567\\
5	0.0132728714180811\\
6	0.0162998960009801\\
7	0.0188416674152903\\
8	0.0189780958604861\\
9	0.0152364305947037\\
10	0.0138095532395181\\
11	0.0151920670476253\\
12	0.0136149394075733\\
13	0.0163340680222918\\
14	0.0160476943609649\\
15	0.010964734276926\\
16	0.0181101722637148\\
17	0.0163784617809373\\
18	0.0180116913029981\\
19	0.0156789396932037\\
20	0.0116350438214403\\
21	0.0134429684856138\\
22	0.0138463567417979\\
23	0.0167334412380073\\
24	0.0118509712232048\\
25	0.0174088936765897\\
26	0.0170817905796831\\
27	0.0138953712568294\\
28	0.0205247944847037\\
29	0.0140818126284471\\
30	0.0170505905325335\\
31	0.0177028382416324\\
32	0.0132840809916283\\
33	0.0150091177189003\\
34	0.0172727872310109\\
35	0.0148387709311013\\
36	0.0184787766254901\\
37	0.0134487729893826\\
38	0.0173210442119884\\
39	0.0143906609046891\\
40	0.0115183314366056\\
41	0.0156872463109362\\
42	0.0146397674614303\\
43	0.0157733762614001\\
44	0.0185336373786293\\
45	0.0190736392188268\\
46	0.0123598820536111\\
47	0.0114158593569317\\
48	0.0156645071996193\\
49	0.016095303590539\\
50	0.0127442083426223\\
51	0.0181988881845562\\
52	0.015126874478586\\
53	0.0162786561176036\\
54	0.0137555934731068\\
55	0.0159182025503257\\
56	0.0135463227396602\\
57	0.0168022721463903\\
58	0.0162708214045688\\
59	0.0123433125081938\\
60	0.0151232031226331\\
61	0.017412084853589\\
62	0.0155451872748412\\
63	0.0137678343707588\\
64	0.0134018189692803\\
65	0.0146783524055392\\
};
\addlegendentry{Approximate root}

\end{axis}
\end{tikzpicture}%
		
		\cprotect\caption{Approximate stochastic square root for $A$ the normalized stochastic matrix obtained from {\verb|Pajek/GD96_c|} of SuiteSparse matrix collection. We depict the difference in the entries of the stationary distribution of the original matrix $A$ and its approximated root.}\label{fig:eigenfailure}
	\end{figure}
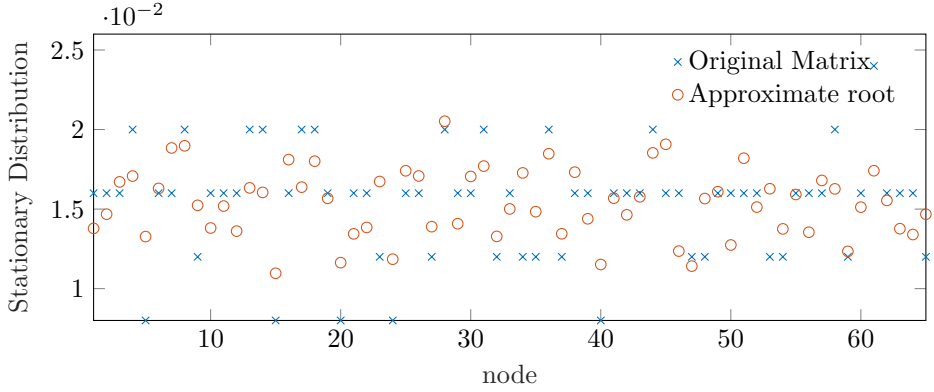
\end{example}

\section{Stochastic \texorpdfstring{$p$th}{pth} root approximation preserving the stationary distribution}\label{sec:stochastic-new-methods}
In this section we first introduce the manifold of positive stochastic matrices, having the same stationary distribution $\boldsymbol{\pi} > 0$, $\boldsymbol{\pi}^T\mathbf{1} = 1$. Then, given a stochastic matrix $A$ with stationary distribution $\boldsymbol{\pi}$,  we approximate its stochastic $p$th root on such a manifold. This way, the stochastic $p$th root approximation of $A$ shares with $A$ the stationary distribution.

\subsection{A new Riemannian manifold}
Let $\boldsymbol{\pi} \in \mathbb{R}^{n}$ be a positive vector such that $\boldsymbol{\pi}^T \mathbf{1} = 1$, and
define the set
\[
\mathbb{S}_n^{\boldsymbol{\pi}} = \{ S \in \mathbb{R}^{n \times n} \,:\; S \mathbf{1} = \mathbf{1}, \; \boldsymbol{\pi}^T S = \boldsymbol{\pi}^T, \; S > 0 \},
\]
i.e., $\mathbb{S}_n^{\boldsymbol{\pi}}$ is the set  of $n\times n$ positive stochastic matrices, having the same stationary distribution~$\boldsymbol{\pi}$.
{After proving that $\mathbb{S}_n^{\boldsymbol{\pi}}$ is a manifold, the analogous of problem~\ref{alg:a2} is rewritten as:
\begin{equation*}%
	\text{Given }A \in \mathbb{S}_n^0 \text{ and } \boldsymbol{\pi} \text{ s.t. } \boldsymbol{\pi}^T A = \boldsymbol{\pi}^T \text{ find } X = 
	\displaystyle \arg\min_{ X  \in  \mathbb{S}_n^{\boldsymbol{\pi}}} \frac{1}{2} \| X^p - A\|_F^2.
\end{equation*}}
By following the approach used in~\cite{Douik8861409} for the manifold of doubly stochastic matrices,  we may prove that  $\mathbb{S}_n^{\boldsymbol{\pi}}$ is an embedded manifold of $\mathbb{R}^{n \times n}$ of dimension $(n-1)^2$, since it is indeed generated by $2n-1$ linearly independent equations.

The following result characterizes the tangent space:

\begin{lemma}\label{lem:tangent-space}
	The tangent space  to $\mathbb{S}_n^{\boldsymbol{\pi}}$ at $S\in\mathbb{S}_n^{\boldsymbol{\pi}}$ is given by
	\begin{equation}\label{eq:tangent_space}
		\mathcal{T}_S \mathbb{S}_n^{\boldsymbol{\pi}} = \{ \xi_S \in \mathbb{R}^{n \times n}\,:\; \xi_S \mathbf{1} = \mathbf{0}, \; \boldsymbol{\pi}^T \xi_S = \mathbf{0} \}.
	\end{equation}
\end{lemma}

\begin{proof}
	Let $S(t)$ be any smooth curve such that $S(0) = S$ and $S(t)\in \mathbb{S}_n^{\boldsymbol{\pi}}$ for $t$ in a  neighborhood of the origin. By differentiating, we find
	\begin{equation*}
	\begin{split}
	\phantom{\boldsymbol{\pi}}S(t)\mathbf{1} = \mathbf{1} \, \Rightarrow \, \phantom{\boldsymbol{\pi}^T}\dot{S}(t)\mathbf{1} = \mathbf{0},\\
	\boldsymbol{\pi}^T S(t)\phantom{\mathbf{1}} = \boldsymbol{\pi}^T \, \Rightarrow \, \boldsymbol{\pi}^T \dot{S}(t) \phantom{\mathbf{1}} = \mathbf{0},
\end{split}
	\end{equation*}
	thus we have
	\begin{equation*}
		\mathcal{T}_S \mathbb{S}_n^{\boldsymbol{\pi}} \subseteq \{ \xi_S \in \mathbb{R}^{n \times n}\,:\; \xi_S \mathbf{1} = \mathbf{0}, \; \boldsymbol{\pi}^T \xi_S = \mathbf{0} \}.
	\end{equation*}
	To prove the opposite inclusion, we observe that the space $\{ \xi_S \in \mathbb{R}^{n \times n}\,:\; \xi_S \mathbf{1} = \mathbf{0}, \; \boldsymbol{\pi}^T \xi_S = \mathbf{0} \}$ is defined by $2n-1$ independent equations. Since the whole space has size $n^2$, then the dimension is given by $n^2-(2n-1) = (n-1)^2$ which equals the size of the space of matrices with a given left- and right-eigenvector.  Therefore, the set we defined has the same dimension as the tangent space, so that they coincide.
\end{proof}

The manifold $\mathbb{S}_n^{\boldsymbol{\pi}} $, endowed with the Fisher metric \eqref{eq:fisher-metric}, 
is a Riemannian manifold.

To use any optimization strategy we need an expression for the projection operator on the tangent space
\[
\Pi_S \,:\,\mathbb{R}^{n \times n} \rightarrow \mathcal{T}_S \mathbb{S}_n^{\boldsymbol{\pi}}.
\]
For a $Z \in \mathbb{R}^{n \times n}$ and an $S \in \mathbb{S}_n^{\boldsymbol{\pi}}$, we can express the orthogonal projection by using the decomposition of any ambient vector into
\begin{equation}\label{eq:matrix-decomposition}
	Z = \Pi_S(Z) + \Pi_S^\perp(Z).
\end{equation}

\begin{lemma}\label{lem:orthogonal_complement}
	The orthogonal complement of the tangent space $\mathcal{T}_S \mathbb{S}_n^{\boldsymbol{\pi}}$ has the expression
	\[
	\mathcal{T}_S^\perp \mathbb{S}_n^{\boldsymbol{\pi}} = \{ \xi_S^\perp \in \mathbb{R}^{n\times n} \,:\, \xi_S^\perp = (\boldsymbol{\alpha} \mathbf{1}^T + \boldsymbol{\pi} \boldsymbol{\beta}^T) \odot S \},
	\]
	for some vectors $\boldsymbol{\alpha},\boldsymbol{\beta} \in \mathbb{R}^{n}$.
\end{lemma}

\begin{proof}
	Let $z=(\boldsymbol{\alpha} \mathbf{1}^T + \boldsymbol{\pi} \boldsymbol{\beta}^T) \odot S $, for some vectors $\boldsymbol{\alpha},\boldsymbol{\beta} \in \mathbb{R}^{n}$. Then
	\begin{equation*}
	\begin{split}
		\langle  z , \xi_S  \rangle_S = & \operatorname{Trace}( (z \oslash S) \xi_S^T ) = \operatorname{Trace}( (\boldsymbol{\alpha} \mathbf{1}^T + \boldsymbol{\pi} \boldsymbol{\beta}^T) \xi_S^T ) \\
		= & \operatorname{Trace}(\boldsymbol{\alpha} \mathbf{1}^T \xi_S^T ) + \operatorname{Trace}( \boldsymbol{\pi} \boldsymbol{\beta}^T \xi_S^T) \\
		= & \boldsymbol{\alpha}^T \underbrace{\xi_S \mathbf{1}}_{= \mathbf{0}} + \boldsymbol{\beta}^T \underbrace{\xi_S^T \boldsymbol{\pi}}_{\substack{= (\boldsymbol{\pi}^T\xi_S)^T = \mathbf{0}}} = 0,
	\end{split}
	\end{equation*}
	since for a $\xi_S \in \mathcal{T}_S \mathbb{S}_n^{\boldsymbol{\pi}}$ we have $\xi_S \mathbf{1} = \mathbf{0}$ and $\boldsymbol{\pi}^T \xi_S = \mathbf{0}^T$. Therefore, we have $\langle  z , \xi_S  \rangle_S = 0$, $\forall\,\xi_S \in \mathcal{T}_S \mathbb{S}_n^{\boldsymbol{\pi}}$, i.e., $z\in \mathcal{T}_S^\perp \mathbb{S}_n^{\boldsymbol{\pi}}$. To prove the opposite inclusion we use a dimensionality argument. Let us introduce the non-singular diagonal matrix  $D_{\boldsymbol{\pi}} = \operatorname{diag}(\boldsymbol{\pi})$ and observe that
	\[
	D_{\boldsymbol{\pi}}^{-1} z = D_{\boldsymbol{\pi}}^{-1} (\boldsymbol{\alpha} \mathbf{1}^T + \boldsymbol{\pi} \boldsymbol{\beta}^T) \odot S = (\hat{\boldsymbol{\alpha}} \mathbf{1}^T + \mathbf{1} \boldsymbol{\beta}^T) \odot S, \quad \hat{\boldsymbol{\alpha}} = D_{\boldsymbol{\pi}}^{-1}\boldsymbol{\alpha}.
	\]
	Then, the quantity on the right depends only on the first row and column of $(\hat{\boldsymbol{\alpha}} \mathbf{1}^T + \mathbf{1} \boldsymbol{\beta}^T)$, since 
	\[
	(D_{\boldsymbol{\pi}}^{-1} z)_{i,j} = \left( \frac{(D_{\boldsymbol{\pi}}^{-1} z)_{i,1}}{S_{i,1}} + \frac{(D_{\boldsymbol{\pi}}^{-1} z)_{1,j}}{S_{1,j}} - \frac{(D_{\boldsymbol{\pi}}^{-1} z)_{1,1}}{S_{1,1}}  \right) \odot S_{i,j}.
	\]
	Therefore, the dimension of the orthogonal complement of the tangent space has the correct dimension $2n-1$. 
\end{proof}

The above result allows us to give an expression for the orthogonal projection { with respect to the scalar product induced by Fisher's metric~\eqref{eq:fisher-metric}}:

\begin{proposition}
	The orthogonal projection $\Pi_S\,:\,\mathbb{R}^{n\times n} \rightarrow \mathcal{T}_S \mathbb{S}_n^{\boldsymbol{\pi}}$ of a matrix $Z${--with respect to the scalar product induced by Fisher's metric~\eqref{eq:fisher-metric}--}has the following expression:
	\[
	\Pi_S(Z) = Z - (\boldsymbol{\alpha} \mathbf{1}^T + \boldsymbol{\pi}\boldsymbol{\beta}^T)\odot S,
	\]
	where the vectors $\boldsymbol{\alpha}$ and $\boldsymbol{\beta}$ are a solution to the following consistent linear system
	\begin{equation}\label{eq:alpha_and_beta_values}
		\begin{bmatrix}
			Z\mathbf{1} \\
			Z^T \boldsymbol{\pi}
		\end{bmatrix}=
		\begin{bmatrix}
			I & D_{\boldsymbol{\pi}} S \\
			S^T D_{\boldsymbol{\pi}} & \mathrm{diag}(S^T D_{\boldsymbol{\pi}} \boldsymbol{\pi}) 
		\end{bmatrix} 
		\begin{bmatrix}
			\boldsymbol{\alpha}
			\\
			\boldsymbol{\beta}
		\end{bmatrix}, \quad D_{\boldsymbol{\pi}} = \operatorname{diag}(\boldsymbol{\pi}).
	\end{equation}
\end{proposition}

\begin{proof}
	The formula for the orthogonal projection follows from Lemma~\ref{lem:orthogonal_complement}. To find an expression for the vectors $\boldsymbol{\alpha}$ and $\boldsymbol{\beta}$, we 
	use~\eqref{eq:matrix-decomposition} and obtain
	\[
	Z\mathbf{1} = \Pi_S(Z)\mathbf{1} + \Pi_S^\perp(Z)\mathbf{1} = \Pi_S^\perp(Z)\mathbf{1}.
	\]
	From Lemma~\ref{lem:orthogonal_complement}, we find that
	\[
	Z \mathbf{1} = ( (\boldsymbol{\alpha}\mathbf{1}^T + \boldsymbol{\pi}\boldsymbol{\beta}^T) \odot S ) \mathbf{1},
	\]
	that is
	\begin{equation*}
		\sum_{j=1}^{n} Z_{i,j} =  \sum_{j=1}^{n} (\alpha_i + \pi_i \beta_j) S_{i,j}= \alpha_i \sum_{j=1}^{n} S_{i,j} + \pi_i \sum_{j=1}^{n} S_{i,j}\beta_j, 
	\end{equation*}
	i.e., in matrix form,
	\[
	Z \mathbf{1} =  \boldsymbol{\alpha} +  D_{\boldsymbol{\pi}} S \boldsymbol{\beta}, \qquad D_{\boldsymbol{\pi}} = \operatorname{diag}(\boldsymbol{\pi}).
	\]
	Similarly, by transposing \eqref{eq:matrix-decomposition}, we obtain
	\begin{equation*}
		\begin{split}
			Z^T \boldsymbol{\pi} = &\; ( (\boldsymbol{\alpha}\mathbf{1}^T + \boldsymbol{\pi}\boldsymbol{\beta}^T) \odot S )^T \boldsymbol{\pi}\\
			=&\;   ( (\mathbf{1}
			\boldsymbol{\alpha}^T) \odot S^T)\boldsymbol{\pi}
			+(\boldsymbol{\beta} \boldsymbol{\pi} ^T) \odot S^T) \boldsymbol{\pi}.  
		\end{split}
	\end{equation*}
	From the properties of the Hadamard product, we find that
	\[
	( (\mathbf{1}
	\boldsymbol{\alpha}^T) \odot S^T)\boldsymbol{\pi}=D_{\mathbf{1}} S^T D_{\boldsymbol{\alpha}} \boldsymbol{\pi}=
	S^T D_{\boldsymbol{\pi}} \boldsymbol{\alpha},
	\]
	and
	\[
	((\boldsymbol{\beta} \boldsymbol{\pi} ^T) \odot S^T) \boldsymbol{\pi}=
	\mathrm{diag}(\boldsymbol{\beta} \boldsymbol{\pi}^T D_{\boldsymbol{\pi}}S)= \mathrm{diag}( S^TD_{\boldsymbol{\pi}}\boldsymbol{\pi}) \boldsymbol{\beta}.
	\]
	Therefore, we conclude that
	\[
	Z^T \boldsymbol{\pi} = 
	S^T D_{\boldsymbol{\pi}} \boldsymbol{\alpha} +
	\mathrm{diag}( S^TD_{\boldsymbol{\pi}}\boldsymbol{\pi}) \boldsymbol{\beta},
	\]
	so that the vectors $\boldsymbol{\alpha}$ and $\boldsymbol{\beta}$ can be found as a solution of the linear system
	\[
	\begin{bmatrix}
		Z\mathbf{1} \\
		Z^T \boldsymbol{\pi}
	\end{bmatrix}=
	\begin{bmatrix}
		I & D_{\boldsymbol{\pi}} S \\
		S^T D_{\boldsymbol{\pi}} & \mathrm{diag}( S^TD_{\boldsymbol{\pi}}\boldsymbol{\pi}) 
	\end{bmatrix} 
	\begin{bmatrix}
		\boldsymbol{\alpha}
		\\
		\boldsymbol{\beta}
	\end{bmatrix}.
	\]
	The linear system is consistent with an affine space of solutions of dimension one, since
	\[
	\begin{bmatrix}
		Z\mathbf{1} \\
		Z^T \boldsymbol{\pi}
	\end{bmatrix}^T \begin{bmatrix}
		- \boldsymbol{\pi} \\
		\mathbf{1}
	\end{bmatrix} = - \mathbf{1}^T Z^T \boldsymbol{\pi} + \boldsymbol{\pi}^T Z \mathbf{1} = - \mathbf{1}^T \boldsymbol{\pi} + \boldsymbol{\pi}^T \mathbf{1} = -1 + 1 = 0,
	\]
	and 
	\[
	\begin{bmatrix}
		I & D_{\boldsymbol{\pi}} S \\
		S^T D_{\boldsymbol{\pi}} & \mathrm{diag}( S^TD_{\boldsymbol{\pi}}\boldsymbol{\pi}) 
	\end{bmatrix} \begin{bmatrix}
		- \boldsymbol{\pi} \\
		\mathbf{1}
	\end{bmatrix} = \begin{bmatrix}
		\mathbf{0} \\
		\mathbf{0}
	\end{bmatrix}.
	\]
\end{proof}

Let now $f : \mathbb{S}_n^{\boldsymbol{\pi}} \rightarrow \mathbb{R}$ be a smooth real function defined on the manifold, and denote by $\operatorname{Grad} f(S)$ its euclidean gradient with respect to the euclidean metric. Then we can express the Riemannian gradient as follows.
\begin{proposition}[Riemannian gradient]\label{pro:riemannian-gradient}
	The Riemannian gradient $\operatorname{grad}f(S)$ is expressed in terms of the Euclidean gradient $\operatorname{Grad}f(S)$ as:
	\begin{equation}\label{eq:riemannian-gradient-on-the-new-manifold}
		\operatorname{grad}f(S) = \Pi_S(\operatorname{Grad}f(S) \odot S).
	\end{equation}
\end{proposition}

\begin{proof}
	We prove \eqref{eq:riemannian-gradient-on-the-new-manifold} by directly applying the Definition~\ref{def:riemannian-gradient-and-hessian} of {the} Riemannian gradient. Indeed, the Riemannian gradient is the \emph{unique} element of $\mathcal{T}_S \mathbb{S}_n^{\boldsymbol{\pi}}$ for which  the directional derivative and the Riemannian metric satisfy the equation
	\begin{align}
		\langle \operatorname{grad}f(S), \xi_S \rangle_S = \mathrm{D} f(S) [\xi_S],\ \forall \ \xi_S \in \mathcal{T}_S \mathbb{S}_n^{\boldsymbol{\pi}}.
		\label{eq:riemannian-gradient-definition}
	\end{align}
	{T}herefore if we find an element of $\mathcal{T}_S \mathbb{S}_n^{\boldsymbol{\pi}}$ for which~\eqref{eq:riemannian-gradient-definition} holds for all tangent vectors, then this is the Riemannian gradient. We start by writing the Euclidean gradient in terms of the directional derivative and the Euclidean scalar product as:
	\[\langle \operatorname{Grad}f(S), \xi \rangle = \mathrm{D} f(S) [\xi],\ \forall \ \xi \in \mathbb{R}^{n \times n}.\]
	Restricting the previous to $\mathcal{T}_S \mathbb{S}_n^{\boldsymbol{\pi}} \subset \mathbb{R}^{n \times n}$, and changing the inner product to the one induced by the Fisher metric~\eqref{eq:fisher-metric} we find
	\[
	\langle \operatorname{Grad}f(S) , \xi_S \rangle = \langle \operatorname{Grad}f(S) \odot S, \xi_S \rangle_S 
	= \mathrm{D} f(S) [\xi_S],\ \forall \ \xi_S \in \mathcal{T}_S \mathbb{S}_n^{\boldsymbol{\pi}}. 
	\]
	To reach the conclusion we now need to apply again Lemma~\ref{lem:orthogonal_complement} and~\eqref{eq:matrix-decomposition} to project the (scaled) Euclidean gradient
	\[ \operatorname{Grad}f(S) \odot S = \Pi_S(\operatorname{Grad}f(S) \odot S) + \Pi_S^\perp(\operatorname{Grad}f(S) \odot S).\]
	From this we find
	\[
	\langle \operatorname{Grad}f(S) \odot S, \xi_S \rangle_S = \langle \Pi_S(\operatorname{Grad}f(S) \odot S), \xi_S \rangle_S,
	\]
	having canceled out the second term $\langle \Pi^\perp_A(\operatorname{Grad}f(S) \odot S), \xi_S \rangle_S = 0$ by means of the definition of the orthogonal complement.
	Summarizing, we have thus shown that $\Pi_S(\operatorname{Grad}f(S) \odot S)$ is \emph{a} tangent vector that satisfies the condition~\eqref{eq:riemannian-gradient-definition}, that permits us to conclude by the uniqueness of the Riemannian gradient that:
	\begin{align*}
		\operatorname{grad} f(S) = \Pi_S \left( \operatorname{Grad} f(S) \odot S \right).
	\end{align*}
\end{proof}

To implement Algorithm~\ref{alg:newton} we also need an expression for the Riemannian Hessian. From Definition~\ref{def:riemannian-gradient-and-hessian}, the Riemannian Hessian is related to the Levi-Civita connection, thus we first need a way of expressing the Levi-Civita connection for the metric~\eqref{eq:fisher-metric}.

\begin{proposition}[Koszul formula, {\cite[Theorem~5.3.1]{AbsilBook}}]\label{pro:Koszul}
	The Levi-Civita connection (Definition~\ref{def:affine-connection}) on the Euclidean space $\mathbb{R}^{n \times n}$ endowed with the Fisher information metric~\eqref{eq:fisher-metric} is given by
	\begin{align*}
		\nabla_{\eta_S} \xi_S = \mathrm{D}(\xi_S)[\eta_S] - \cfrac{1}{2} (\eta_S \odot \xi_S) \oslash S
	\end{align*}
\end{proposition}

\begin{theorem}[Riemannian Hessian]
	The Riemannian Hessian $\mathrm{hess} f(S)[\xi_S]$ can be obtained from the Euclidean gradient $\operatorname{Grad} f(S)$ and the Euclidean Hessian $\operatorname{Hess} f(S)$ by using the identity
	\[
	\mathrm{hess} f(S)[\xi_S] = \Pi_S(\mathrm{D}(\mathrm{grad} f(S))[\xi_S])-\frac12 \Pi_S(( \Pi_S(\mathrm{Grad} f(S))\odot S) \odot \xi_S) \oslash S),
	\]
	where
	\[
	D(\mathrm{grad} f(S))[\xi_S] = \dot{\gamma}[\xi_S]- (\dot{\boldsymbol{\alpha}}[\xi_S] \mathbf{1}^T+\boldsymbol{\pi}\dot{\boldsymbol{\beta}}^T[\xi_S])\odot S -
	(\boldsymbol{\alpha} \mathbf{1}^T+\boldsymbol{\pi}\boldsymbol{\beta}^T)\odot \xi_S,
	\]
	and
	\begin{equation*}
	\begin{split}
		\gamma = & \;  \mathrm{Grad} f(S)\odot S,\\
		\dot{\gamma}[\xi_S] = & \; \mathrm{Hess}\; f(S)[\xi_S]\odot S + \mathrm{Grad}\; f(S)\odot \xi_S,\\
		\mathcal{A} = &\;     \begin{bmatrix}
			I & D_{\boldsymbol{\pi}} S \\
			S^T D_{\boldsymbol{\pi}} & \mathrm{diag}( S^TD_{\boldsymbol{\pi}}\boldsymbol{\pi})
		\end{bmatrix}, \\
		\boldsymbol{\alpha}, \boldsymbol{\beta} \,\text{s.t.}\,&\; \mathcal{A} \begin{bmatrix}
			\boldsymbol{\alpha} \\ \boldsymbol{\beta}
		\end{bmatrix} = \begin{bmatrix}
			\gamma \mathbf{1} \\ \gamma^T \boldsymbol{\pi}
		\end{bmatrix},\\
		\dot{\boldsymbol{\alpha}}[\xi_S], \dot{\boldsymbol{\beta}}[\xi_S] \,\text{s.t.}\,&\; \mathcal{A} \begin{bmatrix}
			\dot{\boldsymbol{\alpha}}[\xi_S] \\ \dot{\boldsymbol{\beta}}[\xi_S]
		\end{bmatrix} = \begin{bmatrix}
			\dot{\gamma}[\xi_S] \mathbf{1} \\ \dot{\gamma}^T[\xi_S] \boldsymbol{\pi}
		\end{bmatrix} -
		\begin{bmatrix}
			0 & D_{\boldsymbol{\pi}} \xi_S \\
			\xi_S^T D_{\boldsymbol{\pi}} & \mathrm{diag}(\xi_ S^TD_{\boldsymbol{\pi}}\boldsymbol{\pi})
		\end{bmatrix} 
		\begin{bmatrix}
			\boldsymbol{\alpha} \\ \boldsymbol{\beta}
		\end{bmatrix}.
	\end{split}
\end{equation*}
\end{theorem}

\begin{proof}
	To obtain the first identity it is sufficient to use the Koszul formula (Proposition~\ref{pro:Koszul})
	\begin{equation*}
		\begin{split}
			\mathrm{hess} f(S)[\xi_S] & =\Pi_S(\mathrm{D}(\mathrm{grad} f(S))[\xi_S])-\frac12 \Pi_S(( \mathrm{grad} f(S))\odot \xi_S)\oslash S)\\
			& = \Pi_S(\mathrm{D}(\mathrm{grad} f(S))[\xi_S]) \\ & \qquad -\frac12 \Pi_S(( \Pi_S(\mathrm{Grad} f(S))\odot S) \odot \xi_S) \oslash S).
		\end{split}
	\end{equation*}
	Then, we need to find an expression for $\mathrm{D}(\mathrm{grad} f(S))[\xi_S]$. To find it, we denote $\gamma= \mathrm{Grad} f(S)\odot S$;  from the properties of the Fréchet derivative we find that
	\begin{equation*}
	\begin{split}
		\mathrm{D}(\mathrm{grad} f(S))[\xi_S]& =\mathrm{D}(\Pi_S(\gamma))[\xi_S]=\mathrm{D}(\gamma-(\boldsymbol{\alpha} \mathbf{1}^T+\boldsymbol{\pi}\boldsymbol{\beta}^T)\odot S)[\xi_S]=\\
		& \mathrm{D}(\gamma)[\xi_S]-\mathrm{D}( (\boldsymbol{\alpha} \mathbf{1}^T+\boldsymbol{\pi}\boldsymbol{\beta}^T)\odot S)[\xi_S]=\\
		& \dot{\gamma}[\xi_S]- (\dot{\boldsymbol{\alpha}}[\xi_S] \mathbf{1}^T+\boldsymbol{\pi}\dot{\boldsymbol{\beta}}^T[\xi_S])\odot S -
		(\boldsymbol{\alpha} \mathbf{1}^T+\boldsymbol{\pi}\boldsymbol{\beta}^T)\odot \xi_S.
	\end{split}
	\end{equation*}
	Now we need an expression for $\dot{\gamma}[\xi_S]$, $\dot{\boldsymbol{\alpha}}[\xi_S]$, and $\dot{\boldsymbol{\beta}}[\xi_S]$. Concerning  $\dot{\gamma}[\xi_S]=\mathrm{D}(\gamma)[\xi_S]$, we have
	\begin{equation*}
		\begin{split}
			\dot{\gamma}[\xi_S]&=\mathrm{D}(\mathrm{Grad}\; f(S))[\xi_S]\odot S + \mathrm{Grad}\; f(S)\odot \xi_S\\
			&=\mathrm{Hess}\; f(S)[\xi_S]\odot S + \mathrm{Grad}\; f(S)\odot \xi_S.
		\end{split}
	\end{equation*}
	To find an expression for $\dot{\boldsymbol{\alpha}}[\xi_S]$ and $\dot{\boldsymbol{\beta}}[\xi_S]$, we compute the derivative along the direction $\xi_S$ of both sides  of the linear system
	\[
	\begin{bmatrix}
		\gamma \mathbf{1} \\ \gamma^T \boldsymbol{\pi}
	\end{bmatrix} =
	\begin{bmatrix}
		I & D_{\boldsymbol{\pi}} S \\
		S^T D_{\boldsymbol{\pi}} & \mathrm{diag}( S^TD_{\boldsymbol{\pi}}\boldsymbol{\pi})
	\end{bmatrix} 
	\begin{bmatrix}
		\boldsymbol{\alpha} \\ \boldsymbol{\beta}
	\end{bmatrix} \equiv \mathcal{A} \begin{bmatrix}
		\boldsymbol{\alpha} \\ \boldsymbol{\beta}
	\end{bmatrix}.
	\]
	Therefore, we obtain
	\[
	\begin{bmatrix}
		\dot{\gamma}[\xi_S] \mathbf{1} \\ \dot{\gamma}^T[\xi_S] \boldsymbol{\pi}
	\end{bmatrix} =
	\begin{bmatrix}
		0 & D_{\boldsymbol{\pi}} \xi_S \\
		\xi_S^T D_{\boldsymbol{\pi}} & \mathrm{diag}(\xi_ S^TD_{\boldsymbol{\pi}}\boldsymbol{\pi})
	\end{bmatrix} 
	\begin{bmatrix}
		\boldsymbol{\alpha} \\ \boldsymbol{\beta}
	\end{bmatrix}+
	\begin{bmatrix}
		I & D_{\boldsymbol{\pi}} S \\
		S^T D_{\boldsymbol{\pi}} & \mathrm{diag}( S^TD_{\boldsymbol{\pi}}\boldsymbol{\pi})
	\end{bmatrix} 
	\begin{bmatrix}
		\dot{\boldsymbol{\alpha}}[\xi_S] \\ \dot{\boldsymbol{\beta}}[\xi_S]
	\end{bmatrix}, 
	\]
	so that $\dot{\boldsymbol{\alpha}}[\xi_S]$ and $\dot{\boldsymbol{\beta}}[\xi_S]$ can be computed by solving the linear system (with the same system matrix)
	\begin{align*}
		\mathcal{A}
		\begin{bmatrix}
			\dot{\boldsymbol{\alpha}}[\xi_S] \\ \dot{\boldsymbol{\beta}}[\xi_S]
		\end{bmatrix}= \begin{bmatrix}
			\dot{\gamma}[\xi_S] \mathbf{1} \\ \dot{\gamma}^T[\xi_S] \boldsymbol{\pi}
		\end{bmatrix} -
		\begin{bmatrix}
			0 & D_{\boldsymbol{\pi}} \xi_S \\
			\xi_S^T D_{\boldsymbol{\pi}} & \mathrm{diag}(\xi_S^TD_{\boldsymbol{\pi}}\boldsymbol{\pi})
		\end{bmatrix} 
		\begin{bmatrix}
			\boldsymbol{\alpha} \\ \boldsymbol{\beta}
		\end{bmatrix} 
	\end{align*}.
\end{proof}

To complete the construction of the Riemannian optimization algorithm, we also need to define the retraction from the tangent bundle to the manifold (Definition~\ref{def:retraction}). To obtain such a map, we apply a suitable modification of the generalized Sinkhorn-Knopp algorithm~\cite{Rothblum}, which is based on the following {theorem}:

\begin{theorem}[Sinkhorn generalization, {\cite[Theorem 2(a)-(b)]{Rothblum}}]\label{thm:sink}
	Let $A \in \mathbb{R}^{n \times n}$ be a nonnegative matrix. Then for any vectors $\mathbf{r},\mathbf{c} \in \mathbb{R}^{n}$ with nonnegative entries there exist diagonal matrices $D_1$ and $D_2$ such that
	\[
	D_1 A D_2 \mathbf{1} = \mathbf{r}, \qquad D_2 A^T D_1 \mathbf{1} = \mathbf{c},
	\]
	if and only if there exists a matrix $B$ such $B\mathbf{1} = \mathbf{r}$ and $B^T\mathbf{1} = \mathbf{c}$, and having the same nonzero pattern as~$A$. Furthermore, if the matrix $A$ is positive, then $D_1$ and $D_2$ are unique up to a constant factor.
\end{theorem}

From the previous result, we can obtain the following generalization which allows us to obtain a matrix on the manifold $\mathbb{S}_n^{\boldsymbol{\pi}}$ through suitable diagonal scaling.
\begin{proposition}\label{prop:extsk}
	Let $A \in \mathbb{R}^{n \times n}$ be a matrix with positive entries. Then  there exist diagonal matrices $D_1$ and $D_2$ such that
	\[
	D_1 A D_2 \mathbf{1} = \mathbf{1}, \qquad \boldsymbol{\pi}^T D_1 A D_2  = \boldsymbol{\pi}^T.
	\]
	Moreover, $D_1$ and $D_2$ are diagonal matrices such that
	$D_1\widehat A D_2 \mathbf{1}=\boldsymbol{\pi}$ and $\mathbf{1}^T D_1\widehat A D_2 =\boldsymbol{\pi}^T$, where $\widehat A=\mathrm{diag}(\boldsymbol{\pi})A$. 
\end{proposition}

\begin{proof}
	Consider the matrix $\widehat A=\mathrm{diag}(\boldsymbol{\pi})A$. 
	By setting $\mathbf{c}=\mathbf{r}=\boldsymbol{\pi}$, according to Theorem~\ref{thm:sink} applied to $\widehat A$, there exist diagonal matrices $D_1$ and $D_2$ such that $D_1\widehat A D_2 \mathbf{1}=\boldsymbol{\pi}$ and $\mathbf{1}^T D_1\widehat A D_2 =\boldsymbol{\pi}^T$.
	Since diagonal matrices commute, from the first equality we obtain
	$\mathrm{diag}(\boldsymbol{\pi})^{-1}D_1\widehat A D_2 \mathbf{1}=\mathbf{1}$, so that
	$D_1 A D_2 \mathbf{1}=\mathbf{1}$; from the second equality, we find that $\mathbf{1}^T \mathrm{diag}(\boldsymbol{\pi}) D_1 A D_2 =\boldsymbol{\pi}^T$, i.e., $\boldsymbol{\pi}^T D_1 A D_2  = \boldsymbol{\pi}^T$.   
\end{proof}

{The above result,
combined with Theorem~\ref{thm:retraction-on-embedded-manifold},
provides an expression for the retraction to the manifold $\mathbb{S}_n^{\boldsymbol{\pi}}$.}
\begin{theorem}[Retraction]
	The map $R: \mathcal{T}\mathbb{S}_n^{\boldsymbol{\pi}} \longrightarrow \mathbb{S}_n^{\boldsymbol{\pi}} $ whose restriction $R_{S}$ to $ \mathcal{T}_S\mathbb{S}_n^{\boldsymbol{\pi}} $ is given by:
	\begin{align}
		R_{S}(\xi_S) = S + \xi_S,
	\end{align}
	is a well-defined retraction on $\mathbb{S}_n^{\boldsymbol{\pi}}$ in the sense of Definition~\ref{def:retraction} whenever $\xi_S$ is in a neighborhood of $\mathbf{0}_{S}$, i.e., whenever $S > - \xi_S$ entry-wise.
	\label{th4}
\end{theorem}

\begin{proof}
	Since $\mathbb{S}_n^{\boldsymbol{\pi}}$ an embedded manifold, we apply Theorem~\ref{thm:retraction-on-embedded-manifold}, where the diffeomorphism $\phi$ is obtained by means of the extension of {the} Sinkhorn theorem given in Proposition~\ref{prop:extsk}.  Since we deal with matrices with positive entries, the result in Proposition~\ref{prop:extsk} is invariant with respect to the scaling $D_1$ and $D_2$. Thus we can assume, without loss of generality, that the first diagonal element $(D_1)_{11}=1$. Then, the map $\phi$ we need to construct is given by
	\begin{align*}
		\phi: \mathbb{S}_n^{\boldsymbol{\pi}} \times \myR^{2n-1} &\longrightarrow \myR^{n \times n} & \begin{array}{l}\myR^{n \times n} = \{ S \in \mathbb{R}^{n \times n} \ : \ S > 0\},\\
			\myR^{2n-1} = \{ \mathbf{x} \in \mathbb{R}^{2n-1} \,:\, \mathbf{x} > 0 \},\end{array} \\[-0.4em]
		\left(S,\begin{pmatrix}
			\mathbf{d}_1 \\ \mathbf{d}_2
		\end{pmatrix}\right) &\longmapsto \operatorname{diag}(1,\mathbf{d}_1) S \operatorname{diag}(\mathbf{d}_2).
	\end{align*}
	Such $\phi$ satisfies all the requirements of Theorem~\ref{thm:retraction-on-embedded-manifold}. Indeed, both $\myR^{2n-1}$ and $\myR^{n \times n}$ are manifold as open subsets of the manifolds $\mathbb{R}^{2n-1}$ and  $\mathbb{R}^{n \times n}$, respectively. Furthermore, they satisfy the dimensionality relation since \[\dim(\mathbb{S}_n^{\boldsymbol{\pi}}) + \dim( \myR^{2n-1}) = (n-1)^2+2n-1 = n^2 = \dim(\mathbb{R}^{n \times n}).\] 
	Finally, the identity element $I \equiv \mathbf{1}$ of $\mathbb{R}^{2n-1}$ satisfies $\phi(S,\mathbf{1}) = S$, and $\phi$ inherits the required regularity from the regularity of the matrix product. To build the projection $\pi_1$ in Theorem~\ref{thm:retraction-on-embedded-manifold} we need the existence of the inverse map $\phi^{-1}$. This {amounts} to an application of the (modified) Sinkhorn-Knopp's algorithm scaling the rows and the columns of the matrix. Observe that this is again a smooth map for the regularity of the matrix product. We have therefore proved that $\phi$ is a diffeomorphism. By Theorem~\ref{thm:retraction-on-embedded-manifold}, this means that $\pi_1(\phi^{-1}(S+\xi_S))$ is a retraction for $\xi_S$ in the neighborhood of $\mathbf{0}_{S}$, i.e., $(S+\xi_S) \in \myR^{n \times n} $ which can explicitly written in an element-wise sense as $S_{ij} > - \xi_S$. Using the definition of $\mathbb{S}_n^{\boldsymbol{\pi}}$ and Lemma~\ref{lem:tangent-space} the inverse map is the identity, since
	\begin{align*}
		(S+\xi_S)\mathbf{1} &= S\mathbf{1}+\xi_S\mathbf{1} = \mathbf{1}+ \mathbf{0} =\mathbf{1}, \\
		(S+\xi_S)^T\boldsymbol{\pi} &= S^T\boldsymbol{\pi}+\xi_S^T\boldsymbol{\pi} = \boldsymbol{\pi} + \mathbf{0} = \boldsymbol{\pi},
	\end{align*}
	hence, the canonical retraction is defined as $R_{S}(\xi_S) = S + \xi_S$.
\end{proof}

To avoid the deterioration of the quality of the analogous retraction on $\mathbb{S}_n^{\mathbf{1}}$ in the presence of small modulus elements in the iterations of the Riemannian optimization algorithms, in~\cite{Douik8861409} a modification based on the combination of the entry-wise exponential of a matrix and the Sinkhorn-Knopp’s algorithm (Theorem~\ref{thm:sink}) is proposed. We adapt here such proposal to the manifold $\mathbb{S}_n^{\boldsymbol{\pi}}$. 

\begin{theorem}
	The map $\hat{R}: \mathcal{T}\mathbb{S}_n^{\boldsymbol{\pi}} \longrightarrow \mathbb{S}_n^{\boldsymbol{\pi}} $ whose restriction $\hat{R}_{S}$ to $ \mathcal{T}_S\mathbb{S}_n^{\boldsymbol{\pi}} $ is given by:
	\begin{align}\label{eq:exp-retraction}
		\hat{R}_{S}(\xi_S) = \mathcal{S}\left( S \odot \exp(\xi_S \oslash S )\right),
	\end{align}
	is a first-order retraction on $\mathbb{S}_n^{\boldsymbol{\pi}}$, where $\mathcal{S}\left( \cdot \right)$ represents an application of the modified Sinkhorn-Knopp’s algorithm in Proposition~\ref{prop:extsk}, and $\exp(\cdot)$ the entry-wise exponential.
\end{theorem}

\begin{proof}
	We need to show Definition~\ref{def:retraction} is verified. The map~\eqref{eq:exp-retraction} is centered, since for an $S \in \mathbb{S}_n^{\boldsymbol{\pi}}$ 
	\[
	\hat{R}_{S}(0) = \mathcal{S}\left( S \odot \exp(0 \oslash S )\right) = \mathcal{S}\left( S \odot \exp(0)\right) = \mathcal{S}\left( S \right) = S,
	\]
	having selected $D_1 = D_2 = I$ in Proposition~\ref{prop:extsk}. To prove the local rigidity, we need to show that the curve $\gamma_{\xi_s}(\tau) = \hat{R}_S(\tau \xi_S)$ satisfies
	\[
	\left.\frac{\mathrm{d}\gamma_{\xi_s}(\tau) }{\mathrm{d}\tau}\right\rvert_{\tau = 0} = \xi_S, \quad \forall\,\xi_S \in \mathcal{T}_S\mathbb{S}_n^{\boldsymbol{\pi}}.
	\]
	By definition
	\[
	\left.\frac{\mathrm{d}\gamma_{\xi_s}(\tau) }{\mathrm{d}\tau}\right\rvert_{\tau = 0} = \lim_{\tau \rightarrow 0} \frac{\mathcal{S}\left( S \odot \exp(\tau \xi_S \oslash S )\right) - S}{\tau} = \lim_{\tau \rightarrow 0} \frac{\mathcal{S}( S + \tau \xi_S + \mathcal{O}(\tau^2)) - S}{\tau},
	\]    
	where the last equality follows from the first order Taylor expansion of the exponential
	\[
	S \odot \exp(\tau \xi_S \oslash S ) = S \odot \left( 1 + \tau \xi_S \oslash S + \mathcal{O}(\tau^2) \right) = S + \tau \xi_S + \mathcal{O}(\tau^2)\text{ for } \tau \rightarrow 0.
	\]
	As in the proof of Theorem~\ref{thm:retraction-on-embedded-manifold}, we can now select $\tau$ small enough for having $S + \tau \xi_S$ a matrix with all positive entries, and apply Proposition~\ref{prop:extsk} to write
	\begin{equation*}
	\begin{split}
	\mathcal{S}( S + \tau \xi_S) = & \, (D_1 + \delta D_1) ( S + \tau \xi_S ) (D_2 + \delta D_2) \\ = & \, D_1 S D_2 + D_1 ( \tau \xi_S ) D_2 + \delta D_1 S D_2 + D_1 S \delta D_2,
	\end{split}
	\end{equation*}
	since $S \in \mathbb{S}_n^{\boldsymbol{\pi}}$ we have $D_1 = D_2 = I$, hence
	\[
	\mathcal{S}( S + \tau \xi_S) =  S  + \tau \xi_S  + \delta D_1 S  + S \delta D_2.
	\]
	We exploit now that $\xi_S \in \mathcal{T}_S\mathbb{S}_n^{\boldsymbol{\pi}}$ (Lemma~\ref{lem:tangent-space}), and write
	\begin{equation*}
	\begin{split}
		\mathbf{1} \equiv \mathcal{S}( S + \tau \xi_S)\mathbf{1} = & \,  S\mathbf{1}  + \tau \xi_S\mathbf{1}  + \delta D_1 S \mathbf{1}  + S \delta D_2 \mathbf{1} \\ 
		= & \, \mathbf{1}  + \delta D_1 \mathbf{1} + S \delta D_2 \mathbf{1} \\
		\boldsymbol{\pi}^T \equiv \boldsymbol{\pi}^T \mathcal{S}( S + \tau \xi_S) = & \,  \boldsymbol{\pi}^TS  + \tau \boldsymbol{\pi}^T\xi_S  + \boldsymbol{\pi}^T \delta D_1 S  + \boldsymbol{\pi}^T S \delta D_2 \\ 
		= & \, \boldsymbol{\pi}^T  + \boldsymbol{\pi}^T \delta D_1 S  + \boldsymbol{\pi}^T \delta D_2,
	\end{split}
	\end{equation*}
	equivalently
	\[
	\hat{\mathcal{M}} \begin{bmatrix}
		\boldsymbol{\delta}_1 \\
		\boldsymbol{\delta_2} 
	\end{bmatrix} \equiv \begin{bmatrix}
		I & S \\
		S^T D_{\boldsymbol{\pi}} & D_{\boldsymbol{\pi}}
	\end{bmatrix} \begin{bmatrix}
		\delta D_1 \mathbf{1} \\
		\delta D_2 \mathbf{1}
	\end{bmatrix} = \begin{bmatrix}
		\mathbf{0} \\
		\mathbf{0}
	\end{bmatrix}.
	\]
	The null space of $\hat{\mathcal{M}}$ is generated by the $[\mathbf{1}^T,-\mathbf{1}^T]^T$ vector, equivalently $\boldsymbol{\delta}_1 = - \boldsymbol{\delta}_2 = c \mathbf{1}$, hence $\delta D_1 S  + S \delta D_2 = 0$. Therefore, we have just proved that $\mathcal{S}(S + \tau \xi_S + \mathcal{O}(\tau^2)) = S + \tau \xi_S + \mathcal{O}(\tau^2)$, and consequently
	\[
	\left.\frac{\mathrm{d}\gamma_{\xi_s}(\tau) }{\mathrm{d}\tau}\right\rvert_{\tau = 0} = \lim_{\tau \rightarrow 0} \frac{ S + \tau \xi_S + \mathcal{O}(\tau^2) - S}{\tau} = \xi_S.
	\]
\end{proof}

\subsection{Computational issues}\label{sec:computational_issues}
We have implemented this new Riemannian manifold in a format compatible with the Manopt library \cite{manopt}, i.e., we have produced a Matlab function that outputs a \mintinline{matlab}{struct} variable whose fields implement the operation on the manifold, i.e., the \mintinline{matlab}{function} with prototype
\begin{minted}[bgcolor=bg,fontsize=\small]{matlab}
function M = multinomialfixedstochasticfactory(pi,optionsolve)
%
%
end
\end{minted}
in which the \mintinline{matlab}{optionsolve} contains options concerning the solution of the auxiliary linear systems, that can be selected when instantiating the manifold. The implementation is based on the \verb|multinomialdoublystochasticfactory.m| code by A. Douik and N. Boumal; see~\cite{Douik8861409} and the relevant references in~\cite{manopt}. 

For the computation of the various projections on the tangent space and for the computation of the Riemannian Hessian we have to solve several compatible singular linear systems of the form
\begin{equation}\label{eq:atildes}
	\mathcal{A} \begin{bmatrix}
		\mathbf{x}\\
		\mathbf{y}
	\end{bmatrix} =
	\begin{bmatrix}
		\mathbf{c}\\
		\mathbf{d}
	\end{bmatrix},
	~~~
	\mathcal{A} = \begin{bmatrix}
		I & D_{\boldsymbol{\pi}} S \\
		S^T D_{\boldsymbol{\pi}} & \mathrm{diag}(S^T D_{\boldsymbol{\pi}} \boldsymbol{\pi})
	\end{bmatrix}.
\end{equation}
To this purpose, we want to use an iterative method of the Krylov type to avoid assembling the $2 \times 2$ block matrix. Since the system is symmetrical, an indication of the convergence properties can be obtained starting from the spectral properties of the matrix $\mathcal{A}$. 

\begin{proposition}[Spectral properties]\label{pro:spectral_properties}
	Given $S \in \mathbb{S}_n^{\boldsymbol{\pi}}$, the $2 \times 2$ block matrix
	$\mathcal{A}$  defined in \eqref{eq:atildes}
	is such that
	\begin{itemize}
		\item $\mathcal{A}$ is similar to a singular M-matrix,
		\item $\lambda(\mathcal{A}) \in \{ 0 \} \cup \left[\frac{\left(\delta^*+1 -\sqrt{\delta^*  (\delta^* +4 r^*-2)+1}\right)}{2}, \max\{ 1+\| \boldsymbol{\pi} \|_\infty, 2\| \boldsymbol{\pi} \|_\infty\}\right]$, for
		\[
		r^* =       
		\min_{j=1,\ldots,n} \max_{i = 1,\ldots, n} s_{i,j}, \text{ and } \delta^* = \min_{ \substack{i=1,\ldots,n \\ i \neq k}} \left( S^T D_{\boldsymbol{\pi}} \boldsymbol{\pi} \right)_i,
		\]
		and $k$ is such that $r^* =       
		\min_{i = 1,\ldots, n} s_{i,k}$;
		moreover, if $r^*+\delta^*< 1$, then
		$\lambda(\mathcal{A}) \in \{ 0 \} \cup \left[ \delta^* \left( 1 - \frac{r^*}{1-\delta^*}\right), \max\{ 1+\| \boldsymbol{\pi} \|_\infty, 2\| \boldsymbol{\pi} \|_\infty\}\right]$.
	\end{itemize}
\end{proposition}

\begin{proof}
	The matrix 
	\[
	\tilde{\mathcal{A}}=D^{-1} \mathcal{A} D=
	\begin{bmatrix}
		I & -S \\
		-S^T D_{\boldsymbol{\pi}}^2 & \mathrm{diag}(S^T D_{\boldsymbol{\pi}} \boldsymbol{\pi}) 
	\end{bmatrix}
	, ~~
	D=\begin{bmatrix}
		D_{\boldsymbol{\pi}} & 0 \\ 0 & -I
	\end{bmatrix},
	\]
	is a Z-matrix. Since $\tilde{\mathcal{A}}\mathbf{1}=0$, then $\tilde{\mathcal{A}}$ is a singular M-matrix \cite{bp:book}. In particular, its eigenvalues     have non negative real part \cite{bp:book}. On the other hand, since $\mathcal{A}$ is symmetric, then its eigenvalues are real. Therefore the eigenvalues of $\mathcal{A}$ are $\lambda_1=0\le \lambda_2\le \cdots\le \lambda_{2n}$. Since $S$ is irreducible and $\boldsymbol{\pi}>0$, then the matrix $\tilde{\mathcal{A}}$ is irreducible as well, therefore $\lambda_2>0$. 
	Since 
	\[
	\mathcal{A}\begin{bmatrix}
		\mathbf{1} \\ \mathbf{1}
	\end{bmatrix}=
	\begin{bmatrix}
		\mathbf{1}+\boldsymbol{\pi} \\ \boldsymbol{\pi}+S^T D_{\boldsymbol{\pi}} \boldsymbol{\pi}
	\end{bmatrix} \le \begin{bmatrix}
		\mathbf{1}+\boldsymbol{\pi} \\ 2\boldsymbol{\pi}
	\end{bmatrix}
	\]
	then
	$\| \mathcal{A}\|_\infty\le \max\{ 1+\| \boldsymbol{\pi} \|_\infty, 2\| \boldsymbol{\pi} \|_\infty\}$, which implies \[\lambda_{2n}\le \max\{ 1+\| \boldsymbol{\pi} \|_\infty, 2\| \boldsymbol{\pi} \|_\infty\}.\]

	To obtain a lower-bound we use the eigenvalue interlacing Theorem \cite[Theorem~4.3.6]{MR2978290} for symmetric matrices. Indeed, 
	by removing the $i$-th column and row from $\mathcal{A}$ we obtain a matrix $\mathcal{A}'$ whose eigenvalues are such that
	\[
	0 = \lambda_1 \le  \lambda_1' \leq \lambda_2 \leq \ldots \leq \lambda_{2n-1}' \leq \lambda_{2n}.
	\]
	Therefore, a lower bound to  $\lambda_1'$ gives a lower bound to $\lambda_2$. 
	For the moment, assume that $i=2n$.
	Since $S$ is a positive matrix, then $\mathcal{A}'$ is irreducible, therefore, by using the same arguments as for $\mathcal{A}$ to show that $\lambda_1>0$, we deduce that $\lambda_1'>0$. 
	To give a positive lower bound to $\lambda_1'$, consider the parametric similarity transformation for $\alpha \in (0,1)$
	\[
	\hat{\mathcal{A}} = \begin{bmatrix}
		\alpha^{-1} D_{\boldsymbol{\pi}}^{-1} & \\
		& I_{n-1}
	\end{bmatrix} \mathcal{A}' \begin{bmatrix}
		\alpha D_{\boldsymbol{\pi}} & \\
		& I_{n-1}
	\end{bmatrix} = \begin{bmatrix}
		I_n & \alpha^{-1} \hat{S} \\
		\alpha \hat{S}^T D_{\boldsymbol{\pi}}^2 & \mathrm{diag}(\hat{S}^T D_{\boldsymbol{\pi}} \boldsymbol{\pi}) 
	\end{bmatrix},
	\]
	where $\hat{S}$ is the $n\times(n-1)$ matrix obtained by removing the last column {of} the matrix $S$. Let us call $\boldsymbol{\delta} = \hat{S}^T D_{\boldsymbol{\pi}} \boldsymbol{\pi}$ the $n-1$-vector generating the diagonal matrix in the $(2,2)$ block, then we can estimate the smallest eigenvalue by Gershgorin Theorem~\cite[Chapter~6]{MR2978290}, that is, we can estimate $\lambda_1'$ by the intersection with the $x$ axis of the left-most circle. The circles for the first $n$ rows have center in $1$ and radii 
	\[
	\alpha^{-1} r_i = \alpha^{-1} (\hat{S} \mathbf{1})_i = \alpha^{-1} \left( \mathbf{1} - S \mathbf{e}_n \right)_i = \alpha^{-1} (1-s_{in}) < \alpha^{-1} 1, \; i = 1,\ldots,n,
	\]
	whilst the circles for the last $n-1$ rows have center $\delta_i$ and radii $\alpha \delta_i$. Thus we have  
	\[
	\lambda_1' > \min \left\lbrace 1 - \alpha^{-1} r, \delta (1-\alpha) \right\rbrace
	\]
	for 
	\[
	r = \alpha^{-1} \max_{i = 1,\ldots,n} r_i, \qquad \delta = \min_{i = 1,\ldots,n-1} \delta_i.
	\]
	Observe now that, instead of removing the last row and column in $\widetilde{A}$, we can optimize the bound by selecting the column $\hat{j}$ to which corresponds
	\[
	r^* = \min_{j=1,\ldots,n} \max_{i = 1,\ldots, n} (1-s_{i,j}),
	\]
	and the corresponding 
	\[
	\delta^* = \min_{ \substack{i=1,\ldots,n \\ i \neq {\hat j}}} \delta_i.
	\]
	We can then solve the problem by solving the minimization problem\pgfmathsetmacro\MathAxis{height("$\vcenter{}$")}
	\[
	\alpha^* = \arg\max\min_{\alpha \in (0,1)}\left\lbrace 1 - \alpha^{-1} r, \delta (1-\alpha) \right\rbrace = \begin{tikzpicture}[baseline={(0, 1cm-\MathAxis pt)}]
		\begin{axis}[
			width=1.65in,
			height=1.65in,
			axis lines=middle,
			xmin = -0.5,
			xmax = 1.5,
			ymin = -0.5,
			ymax = 1.5,
			ytick={0, 1},
			xtick={0, 1},
			]
			\addplot[domain=0:1.5,dashed] {1};
			\addplot[domain=0:1.5,smooth,samples=200] { 1 - 0.1/x )};
			\addplot[domain=0:1.5,smooth,samples=200] { 0.6*(1-x)};
			\addplot[red,mark=*] coordinates {(0.193713,0.483772)} node[pin=25:{$\alpha$}]{};
			\addplot[mark=-,thick] coordinates {(0.1,0.6)} node[pin=140:{$\delta$}]{};
			\addplot[mark=|,thick] coordinates {(0.1,0)} node[pin=150:{$r$}]{};
			\addplot[domain=0.1:0.193713,smooth,samples=200,thick,blue] { 1 - 0.1/x )};
			\addplot[domain=0.193713:1,smooth,samples=200,thick,blue] { 0.6*(1-x)};
		\end{axis}
	\end{tikzpicture}
	\]
	which is therefore equivalent to obtaining the positive root of the quadratic equation
	\[
	1-\frac{r}{\alpha }=(1-\alpha ) \delta \,\Leftrightarrow\,\alpha^* = \frac{\delta - 1 +\sqrt{\delta ^2+\delta  (4 r-2)+1}}{2 \delta }.
	\]
	The lower-bound is then given by
	\[
	\lambda_2 \geq \frac{1}{2} \left(\delta^*+1 -\sqrt{\delta^*  (\delta^* +4 r^*-2)+1}\right). 
	\]
	We can further elaborate on the above bound.
	Indeed, since $(1+x)^{1/2}\le 1+\frac12 x$, we obtain
	\begin{equation*}
	\begin{split}
		\sqrt{\delta^*  (\delta^* +4 r^*-2)+1}  = & \sqrt{ (1-\delta^*)^2(1+ 4 r^* \delta^* (1-\delta^*)^{-2})}\le \\
		& (1-\delta^*) (1+ 2 r^* \delta^* (1-\delta^*)^{-2}),
	\end{split}
	\end{equation*}
	so that
	\[
	\lambda_2\ge \delta^* \left( 1 - \frac{r^*}{1-\delta^*}\right).
	\]
	This latter inequality makes sense if $r^*+\delta^*< 1$. 
\end{proof}

{In Section~\ref{sec:experiment-on-bounds} we show, through some numerical experiments,  the sharpness of the bounds given in the above proposition.}

The eigenvalue properties proved in Proposition~\ref{pro:spectral_properties} suggest different strategies to solve the linear system \eqref{eq:atildes}. Firstly, we can directly apply the Conjugate Gradient (CG) to the linear system, in fact, if the starting vector is not in the null space of the matrix, we do not undergo a breakdown; some preconditioning strategies are discussed in the sequel. Secondly, we can consider using the LSQR method on the system instead. To reduce the dimensionality of the problem, we can solve the system for the Schur complement with respect to the $(1,1)$-block, i.e. we solve instead
\begin{equation}\label{eq:schurversion}
	[ \mathrm{diag}(S^T D_{\boldsymbol{\pi}} \boldsymbol{\pi}) - S^T D_{\boldsymbol{\pi}^2} S ]\mathbf{y} = \mathbf{c} - S^T D_{\boldsymbol{\pi}} \mathbf{b}, \quad \mathbf{x} = \mathbf{b} - D_{\boldsymbol{\pi}} S \mathbf{y}.
\end{equation}
The matrix in the above system is a symmetric irreducible singular M-matrix, therefore it is  semidefinite, with a simple eigenvalue equal to 0.

In both formulations \eqref{eq:atildes} and \eqref{eq:schurversion}, since the basis of the null space is known, we can consider the de-singularized version obtained through a rank 1 update of the matrix; so that, in small size problems, we can employ the $LU$-factorization to solve the associated linear systems. Specifically, for the $2 \times 2$-block formulation \eqref{eq:atildes}, this consists in solving the updated linear system
\[
\left( \begin{bmatrix}
	I & D_{\boldsymbol{\pi}} S \\
	S^T D_{\boldsymbol{\pi}} & \mathrm{diag}(S^T D_{\boldsymbol{\pi}} \boldsymbol{\pi}) 
\end{bmatrix} + \frac{1}{\boldsymbol{\pi}^T\boldsymbol{\pi} + n} \begin{bmatrix}
	\boldsymbol{\pi}\\-\mathbf{1}
\end{bmatrix} [\boldsymbol{\pi}^T,-\mathbf{1}^T]\right) \begin{bmatrix}
	\hat{\mathbf{x}}\\
	\hat{\mathbf{y}}
\end{bmatrix} = \begin{bmatrix}
	\mathbf{b}\\
	\mathbf{c}
\end{bmatrix}.
\]
Indeed, we may easily observe that $[\hat{\mathbf{x}}^T,
\hat{\mathbf{y}}^T]^T$ solves also the original linear system \eqref{eq:atildes}.
{It is worth to point out that linear systems with matrices of the kind $\mathcal{A}+\gamma U U^T$, where $\mathcal{A}$ is positive semi-definite and $U$ is low rank, arise in various applications (see for instance \cite{BenziFaccio}). In our case the correction is rank-one, and it is made by an eigenvector of $\mathcal{A}$, so that  the eigenvalues of the new matrix  are the eigenvalues of the matrix $\mathcal{A}$, except for the eigenvalue 0 which is replaced by 1. Since, for Proposition~\ref{pro:spectral_properties}, the upper bound on the eigenvalues of $\mathcal{A}$ is
$\max\{ 1+\| \boldsymbol{\pi} \|_\infty, 2\| \boldsymbol{\pi} \|_\infty\}\ge 1$,  we expect that the rank one update does not modify conditioning of the linear system \eqref{eq:atildes}. }

For the Schur complement version \eqref{eq:schurversion}, the updated system is given by
\[
\left[ \mathrm{diag}(S^T D_{\boldsymbol{\pi}} \boldsymbol{\pi}) - S^T D_{\boldsymbol{\pi}^2} S + \sigma\frac{1}{n} \mathbf{1}\mathbf{1}^T \right]\hat{\mathbf{y}} = \mathbf{c} - S^T D_{\boldsymbol{\pi}} \mathbf{b}, \quad \hat{\mathbf{x}} = \hat{\mathbf{b}} - D_{\boldsymbol{\pi}} S \hat{\mathbf{y}}.
\]
where $\sigma>0$. As in the previous case $[\hat{\mathbf{x}}^T,
\hat{\mathbf{y}}^T]^T$ solves also the original linear system \eqref{eq:atildes}.
Moreover, the eigenvalues of the matrix in the above system are the eigenvalues of the matrix $\mathrm{diag}(S^T D_{\boldsymbol{\pi}} \boldsymbol{\pi}) - S^T D_{\boldsymbol{\pi}^2} S$, except for the eigenvalue 0 which is replaced by $\sigma$. In order not to deteriorate the conditioning of the system, the parameter $\sigma$ should be chosen between the smallest nonzero eigenvalue and the largest eigenvalue of the matrix $\mathrm{diag}(S^T D_{\boldsymbol{\pi}} \boldsymbol{\pi}) - S^T D_{\boldsymbol{\pi}^2} S$. By using Gershgorin theorem, an upper bound to the spectral radius is given by $2\| S^T D_{\boldsymbol{\pi}} \boldsymbol{\pi} \|_\infty$, while a lower bound can be found by using the Cauchy interlacing theorem, as in the proof of Proposition~\ref{pro:spectral_properties}, by bounding the smallest eigenvalue of a principal $(n-1)\times (n-1)$ matrix. Therefore $\sigma$ might be chosen as the arithmetic mean between these two bounds.

The different strategies can be selected through the \mintinline[breaklines,breakanywhere]{matlab}{optionsolve} variable in the definition phase of the manifold. Default values can be generated through the \mintinline[breaklines,breakanywhere]{matlab}{initoptions()} function, i.e., \mintinline[breaklines,breakanywhere]{matlab}{optionsolve = initoptions()}. The user can select between the $2\times 2$-block formulation of the linear system (\mintinline[breaklines,breakanywhere]{matlab}{optionsolve.formulation = "block";}) and the reduction to the Schur complement (\mintinline[breaklines,breakanywhere]{matlab}{optionsolve.formulation = "schur";}). In both cases the desingularization strategy via the rank-1 update can be activated (\mintinline[breaklines,breakanywhere]{matlab}{options.correction = true;}). The solution method for the linear systems can then be selected between the direct strategy (\mintinline[breaklines,breakanywhere]{matlab}{options.method = "direct";}), the CG (\mintinline[breaklines,breakanywhere]{matlab}{options.method = "cg";}), and the LSQR method ((\mintinline[breaklines,breakanywhere]{matlab}{options.method = "lsqr";}). All the options are case-insensitive. Additional debugging and tracing options can be enabled through this facility and are discussed in the code.

To precondition the CG algorithm we start from an empirical observation, since the target matrix $X$ has to be stochastic we expect, as the dimension $n$ of the problem grows, to encounter a large number of small entries. This suggests using a \emph{diagonally compensated modified incomplete Cholesky}~\cite[Section~10.3.5]{MR1990645} directly on the system matrix. The diagonal compensation term, to avoid the presence of nonpositive pivots, can be taken to be $\min_{i=1,\ldots,n}\pi_i$. Another strategy consists in first scaling the system
\begin{equation*}
\begin{split}
	\left[ I -  \mathrm{diag}(S^T D_{\boldsymbol{\pi}} \boldsymbol{\pi})^{-\nicefrac{1}{2}} S^T D_{\boldsymbol{\pi}^2} S \mathrm{diag}(S^T D_{\boldsymbol{\pi}} \boldsymbol{\pi})^{-\nicefrac{1}{2}} \right]  \tilde{\mathbf{y}} = \\ \qquad \mathrm{diag}(S^T D_{\boldsymbol{\pi}} \boldsymbol{\pi})^{-\nicefrac{1}{2}}  \left( \mathbf{c} - S^T D_{\boldsymbol{\pi}} \mathbf{b}\right),
\end{split}
\end{equation*}
and then recover 
\[
\mathbf{y} = \mathrm{diag}(S^T D_{\boldsymbol{\pi}} \boldsymbol{\pi})^{-\nicefrac{1}{2}} \tilde{\mathbf{y}},  \quad \mathbf{x} = \mathbf{b} - D_{\boldsymbol{\pi}} S \mathbf{y}.
\]
By calling $\tilde{S} = \mathrm{diag}(S^T D_{\boldsymbol{\pi}} \boldsymbol{\pi})^{-\nicefrac{1}{2}} S^T D_{\boldsymbol{\pi}^2} S \mathrm{diag}(S^T D_{\boldsymbol{\pi}} \boldsymbol{\pi})^{-\nicefrac{1}{2}}$, and observing that $\rho(\tilde{S})$ $=1$, we can use a truncated Neumann series as preconditioner, i.e.,
\begin{equation}\label{eq:neumann_preconditioner}
	P_{k,\tau} = I + \sum_{j = 1}^k \hat{S}^j, \qquad (\hat{S})_{p,q} = \begin{cases}
		(\tilde{S})_{p,q}, & |\tilde{S}_{i,j}| \geq \tau,\\
		0, & \text{otherwise}.
	\end{cases} 
\end{equation}
The two strategies can be selected by choosing \mintinline[breaklines,breakanywhere]{matlab}{options.method = "pcg";} or \mintinline[breaklines,breakanywhere]{matlab}{options.method = "pcg2";} respectively in the code. In both cases, the preconditioner must be regenerated anew at each outer iteration of the optimization method, i.e., whenever we move the tangent space.

{The overall cost of the procedure is dominated by $O(n^3)$ matrix-power operations needed for the computation of the objective function. Furthermore, we always need to store at least a dense matrix of the same size $n$ of the problem. All the proposed approaches for the solution of the auxiliary linear systems reduce the storage cost and require a number of operation that is less than cubic. The MATLAB implementation is expected to work on current standard computer for matrices of size $n$ up to few thousands.}

\section{Numerical Examples}\label{sec:numer_ex}

In this section, we will compare the algorithms based on Riemannian optimization with the methods available in the literature. We will also validate the theoretical results discussed in Section~\ref{sec:computational_issues} concerning the location of eigenvalues. Specifically, in Section~\ref{sec:stochastic-old-methods-experiments} we compare the Riemannian optimization on the $\mathbb{S}_n$ manifold with the methods based on constrained optimization available in the literature. Then, in Section~\ref{sec:stochastic-new-methods-experiments} we test the new approach that preserves the stationary distribution, i.e., the Riemannian optimization routines on the $\mathbb{S}_n^{\boldsymbol{\pi}}$ manifold. Furthermore, we numerically investigate also the computational issues discussed in Section~\ref{sec:computational_issues}.  Finally, in Section~\ref{sec:reducible-chains} we test our algorithms on an application in finance, where $A$ is the transition matrix of a Markov chain which represents the dynamics of the different credit ratings. The peculiarity of this problem is that $A$ is reducible, therefore its invariant distribution $\boldsymbol{\pi}$ is not positive, so the manifold $\mathbb{S}_n^{\boldsymbol{\pi}}$ cannot be defined. 

The numerical examples have been executed on a Laptop with Intel\textsuperscript{\textregistered} Core\textsuperscript{\texttrademark} i7-8750H CPU @ 2.20GHz with 16~Gb of memory, and running MATLAB 2023a and Manopt v.7.1. Code and examples are available in the GitHub repository \href{https://github.com/Cirdans-Home/pth-root-stochastic}{github.com/Cirdans-Home/pth-root-stochastic} and can be reproduced.

\subsection{Stochastic \texorpdfstring{$p$th}{pth} root approximation via Riemannian optimization}\label{sec:stochastic-old-methods-experiments}
In this section, we address the problem of the stochastic \texorpdfstring{$p$th}{pth} root approximation via Riemannian optimization with the manifold described in Section~\ref{sec:stochastic-old-methods}. We compare the Riemannian algorithm with the \emph{interior point method} implemented in the MATLAB routine \mintinline{matlab}{fmincon}~\cite{interiorpoint}. For the construction of the test matrices, we have prepared a generator of stochastic matrices of different classes \mintinline[breaklines,breakanywhere]{matlab}{matrixgenerator(n,p,seed,classes,numbers)} that produces matrices of which we want to approximate the stochastic $p$th root;  the classes are described in the Table~\ref{tab:matrix-classes}.
\begin{table}[htbp]
	\centering
	\setminted[matlab]{fontsize=\footnotesize}
	\caption{Matrix classes produced by the \mintinline{matlab}{matrixgenerator} code. The routine also allows fixing the seed of the random number generator so that the set of matrices generated for equal parameters is always the same.}
	\label{tab:matrix-classes}
	\begin{tabular}{p{3cm}p{3.5cm}ccc}
		\toprule
		Name & Description & Size & Embeddability & Ref.\\
		\midrule
		uniform stochastic & %
\begin{tabminted}{matlab}
B = rand(n,n);
D = diag(sum(B,2));
A = D\B;
\end{tabminted} 
		& $n$ & unknown & \\ \cmidrule{2-4}
		$p$th power of uniform stochastic & %
\begin{tabminted}{matlab}
B = rand(n,n);
D = diag(sum(B,2));
A = mpower(D\B,p);
\end{tabminted} 
		& $n$ & yes & \\ \cmidrule{2-4}
		$\exp$ of intensity matrix & %
\begin{tabminted}{matlab}
B = rand(n,n);
B(1:1+n:end) = 0;
D = diag(sum(B,2));
A = expm(B - D);
\end{tabminted} 
		& $n$ & yes & \\ \cmidrule{2-4}
		K80 & %
\begin{tabminted}{matlab}
b = rand(1);
c = sqrt(b)-b;
a = 1-b-2*c;
A0= [a,b;a,b];
E = ones(2,2);
A = [A0,c*E;c*E,A0];
\end{tabminted} 
		& 4 & yes & \cite{Casanellas2020} \\ \cmidrule{2-4}
		& %
\begin{tabminted}{matlab}
b = 0.5*rand(1);
c = (1-2*b)/2;
a = 1-b-2*c;
A0 = [a,b;a,b];
E = ones(2,2);
A = [A0,c*E;c*E,A0];
\end{tabminted} 
		& 4 & no & \cite{Casanellas2020} \\ \cmidrule{2-4}
		Pei &
\begin{tabminted}{matlab}
I = eye(n,n);
J = ones(n,n);
alpha = rand(1)-(1/(n-1))^p;
beta = (1-alpha)/n;
A = alpha*I+beta*J;
\end{tabminted} 
		& $n$ & yes & \cite{LinThesis} \\
		\bottomrule
	\end{tabular}
\end{table}
To obtain the approximation we use the formulation~\ref{alg:a2} and generate $40$ stochastic test matrices for each class, where the dimension of the matrices with variable size is $n = 100$. As Riemannian optimization algorithms we consider the \mintinline{matlab}{trustregions} algorithm~\cite{genrtr} and the \mintinline{matlab}{rlbfgs} algorithm~\cite{Huang2016}. The \mintinline{matlab}{trustregions} algorithm uses both the cost functional, its gradient, and an approximation of the Hessian. On the other hand, the \mintinline{matlab}{rlbfgs} algorithm uses only the cost functional, and the gradient. To approximate Euclidean gradients and Hessian matrices we use automatic differentiation (AD) instead of defining them analytically. All methods are initialized from the same starting point generated as a random point on the manifold of stochastic matrices $\mathcal{S}_n$. We measure the quality of the obtained results in terms of the residual on the cost functional~\ref{alg:a2} and on the time necessary to carry out the optimization. We report the data in the form of a performance profile in Figure~\ref{fig:standard_results}{; see~\cite[Section~2]{PerformanceProfiles} for the general usage of performance profiles to compare algorithms}.
\begin{figure}[htbp]
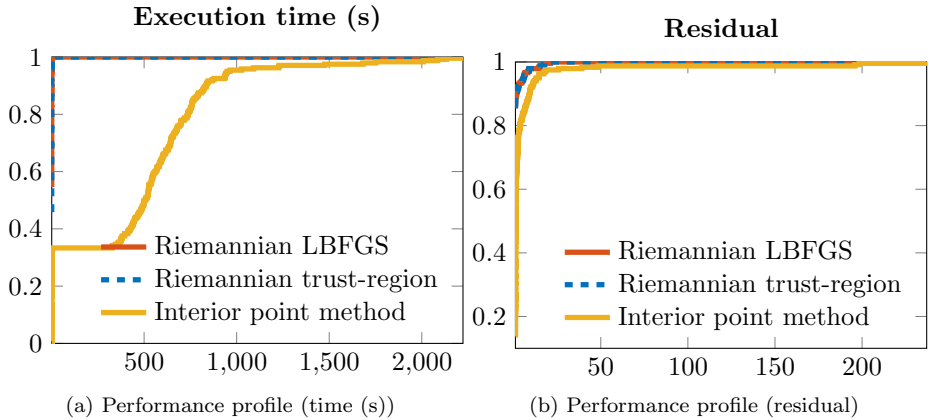

	\centering
	\subfloat[Performance profile (time (s))]{\input{pprofile_residual_stoch}}
	\subfloat[Performance profile (residual)]{\input{pprofile_time_stoch}}
	\caption{Performance comparison for stochastic $p$th root approximation via Riemannian optimization and constrained optimization with the formulation~\ref{alg:a2}. The test matrices are $40$ stochastic matrices from each class in Table~\ref{tab:matrix-classes}.}
	\label{fig:standard_results}
\end{figure}
The results show that the Riemannian algorithms perform much better than their counterpart based on constrained optimization. Both the execution times and the residuals are practically always lower. Furthermore, no significant difference is observed between the use of the {LBFGS and the trust-region algorithm}, i.e., the two lines are practically overlapping and have little difference on an extremely low percentage of problems.

\subsection{Stochastic \texorpdfstring{$p$th}{pth} root approximation preserving the stationary distribution}\label{sec:stochastic-new-methods-experiments}

In this section, we compare the two Riemannian optimization algorithms on the manifolds $\mathbb{S}_n$ and $\mathbb{S}_n^{\boldsymbol{\pi}}$, for the different test matrices from Table~\ref{tab:matrix-classes}. Specifically, we compare the stationary distribution of the approximation $X$ obtained in the two manifolds, with the stationary distribution $\boldsymbol{\pi}$ of $A$, by using the same tolerance request on the gradient norm (\mintinline{matlab}{1e-4}) and the same starting point for the optimization algorithm (\mintinline{matlab}{X0 = M.rand()}) which, in this case, is the \texttt{trustregions} method.

\begin{figure}[htbp]
	\centering
	\subfloat[$p=2$: square root\label{fig:preserving-pi}]{\definecolor{mycolor1}{rgb}{0.00000,0.44700,0.74100}%
\definecolor{mycolor2}{rgb}{0.85000,0.32500,0.09800}%
\begin{tikzpicture}

\begin{axis}[%
width=0.275\columnwidth,
height=1.2in,
at={(0\columnwidth,0in)},
scale only axis,
xmin=1,
xmax=6,
xtick={1,2,3,4,5,6},
xticklabels={{Unif.},{pth power of unif.},{exp of intensity},{K80 (Emb.)},{K80 (Not. Emb.)},{Pei}},
xticklabel style={rotate=90,font=\scriptsize},
ymode=log,
ymin=2.20048160698239e-05,
ymax=0.266865966204034,
yminorticks=true,
ylabel style={font=\color{white!15!black}},
ylabel={$\| X^2 - A \|_F$},
axis background/.style={fill=white},
legend style={at={(1.03,0.5)}, anchor=west, legend cell align=left, align=left, draw=white!15!black}
]
\addplot [color=mycolor1, only marks, mark size=3.5pt, mark=x, mark options={solid, mycolor1}]
  table[row sep=crcr]{%
1	0.256324652067274\\
2	0.0101996845104076\\
3	0.00288243842368615\\
4	0.000168692300399253\\
5	0.00954447899240632\\
6	0.000266727489038104\\
};

\addplot [color=mycolor2, only marks, mark size=3.5pt, mark=o, mark options={solid, mycolor2}]
  table[row sep=crcr]{%
1	0.266865966204034\\
2	0.0113954932238444\\
3	0.000943356234015172\\
4	2.20048160698239e-05\\
5	0.00146056530748389\\
6	0.000151576367967098\\
};

\end{axis}

\begin{axis}[%
width=0.275\columnwidth,
height=1.2in,
at={(0.483\columnwidth,0in)},
scale only axis,
xmin=1,
xmax=6,
xtick={1,2,3,4,5,6},
xticklabels={{Unif.},{pth power of unif.},{exp of intensity},{K80 (Emb.)},{K80 (Not. Emb.)},{Pei}},
xticklabel style={rotate=90,font=\scriptsize},
ymode=log,
ymin=2.08166817117217e-17,
ymax=0.298349169312101,
yminorticks=true,
ylabel style={font=\color{white!15!black}, align=center},
ylabel={Infinity norm error between\\[1ex]target steady state and obtained one},
axis background/.style={fill=white},
legend style={at={(1.03,0.5)}, anchor=west, legend cell align=left, align=left, draw=none, fill=none}
]
\addplot [color=mycolor1, only marks, mark size=3.5pt, mark=x, mark options={solid, mycolor1}]
  table[row sep=crcr]{%
1	5.55111512312578e-17\\
2	2.08166817117217e-17\\
3	2.08166817117217e-17\\
4	5.55111512312578e-17\\
5	5.55111512312578e-17\\
6	2.77555756156289e-17\\
};
\addlegendentry{$S_n^\pi$}

\addplot [color=mycolor2, only marks, mark size=3.5pt, mark=o, mark options={solid, mycolor2}]
  table[row sep=crcr]{%
1	0.0154849912027682\\
2	0.0175258523200822\\
3	0.0162678046483902\\
4	0.298349169312101\\
5	0.00220991811210569\\
6	1.176943888926e-05\\
};
\addlegendentry{$S_n$}

\addplot [color=black, dashed]
  table[row sep=crcr]{%
1	2.22044604925031e-16\\
1.05050505050505	2.22044604925031e-16\\
1.1010101010101	2.22044604925031e-16\\
1.15151515151515	2.22044604925031e-16\\
1.2020202020202	2.22044604925031e-16\\
1.25252525252525	2.22044604925031e-16\\
1.3030303030303	2.22044604925031e-16\\
1.35353535353535	2.22044604925031e-16\\
1.4040404040404	2.22044604925031e-16\\
1.45454545454545	2.22044604925031e-16\\
1.50505050505051	2.22044604925031e-16\\
1.55555555555556	2.22044604925031e-16\\
1.60606060606061	2.22044604925031e-16\\
1.65656565656566	2.22044604925031e-16\\
1.70707070707071	2.22044604925031e-16\\
1.75757575757576	2.22044604925031e-16\\
1.80808080808081	2.22044604925031e-16\\
1.85858585858586	2.22044604925031e-16\\
1.90909090909091	2.22044604925031e-16\\
1.95959595959596	2.22044604925031e-16\\
2.01010101010101	2.22044604925031e-16\\
2.06060606060606	2.22044604925031e-16\\
2.11111111111111	2.22044604925031e-16\\
2.16161616161616	2.22044604925031e-16\\
2.21212121212121	2.22044604925031e-16\\
2.26262626262626	2.22044604925031e-16\\
2.31313131313131	2.22044604925031e-16\\
2.36363636363636	2.22044604925031e-16\\
2.41414141414141	2.22044604925031e-16\\
2.46464646464646	2.22044604925031e-16\\
2.51515151515152	2.22044604925031e-16\\
2.56565656565657	2.22044604925031e-16\\
2.61616161616162	2.22044604925031e-16\\
2.66666666666667	2.22044604925031e-16\\
2.71717171717172	2.22044604925031e-16\\
2.76767676767677	2.22044604925031e-16\\
2.81818181818182	2.22044604925031e-16\\
2.86868686868687	2.22044604925031e-16\\
2.91919191919192	2.22044604925031e-16\\
2.96969696969697	2.22044604925031e-16\\
3.02020202020202	2.22044604925031e-16\\
3.07070707070707	2.22044604925031e-16\\
3.12121212121212	2.22044604925031e-16\\
3.17171717171717	2.22044604925031e-16\\
3.22222222222222	2.22044604925031e-16\\
3.27272727272727	2.22044604925031e-16\\
3.32323232323232	2.22044604925031e-16\\
3.37373737373737	2.22044604925031e-16\\
3.42424242424242	2.22044604925031e-16\\
3.47474747474747	2.22044604925031e-16\\
3.52525252525253	2.22044604925031e-16\\
3.57575757575758	2.22044604925031e-16\\
3.62626262626263	2.22044604925031e-16\\
3.67676767676768	2.22044604925031e-16\\
3.72727272727273	2.22044604925031e-16\\
3.77777777777778	2.22044604925031e-16\\
3.82828282828283	2.22044604925031e-16\\
3.87878787878788	2.22044604925031e-16\\
3.92929292929293	2.22044604925031e-16\\
3.97979797979798	2.22044604925031e-16\\
4.03030303030303	2.22044604925031e-16\\
4.08080808080808	2.22044604925031e-16\\
4.13131313131313	2.22044604925031e-16\\
4.18181818181818	2.22044604925031e-16\\
4.23232323232323	2.22044604925031e-16\\
4.28282828282828	2.22044604925031e-16\\
4.33333333333333	2.22044604925031e-16\\
4.38383838383838	2.22044604925031e-16\\
4.43434343434343	2.22044604925031e-16\\
4.48484848484848	2.22044604925031e-16\\
4.53535353535354	2.22044604925031e-16\\
4.58585858585859	2.22044604925031e-16\\
4.63636363636364	2.22044604925031e-16\\
4.68686868686869	2.22044604925031e-16\\
4.73737373737374	2.22044604925031e-16\\
4.78787878787879	2.22044604925031e-16\\
4.83838383838384	2.22044604925031e-16\\
4.88888888888889	2.22044604925031e-16\\
4.93939393939394	2.22044604925031e-16\\
4.98989898989899	2.22044604925031e-16\\
5.04040404040404	2.22044604925031e-16\\
5.09090909090909	2.22044604925031e-16\\
5.14141414141414	2.22044604925031e-16\\
5.19191919191919	2.22044604925031e-16\\
5.24242424242424	2.22044604925031e-16\\
5.29292929292929	2.22044604925031e-16\\
5.34343434343434	2.22044604925031e-16\\
5.39393939393939	2.22044604925031e-16\\
5.44444444444444	2.22044604925031e-16\\
5.49494949494949	2.22044604925031e-16\\
5.54545454545455	2.22044604925031e-16\\
5.5959595959596	2.22044604925031e-16\\
5.64646464646465	2.22044604925031e-16\\
5.6969696969697	2.22044604925031e-16\\
5.74747474747475	2.22044604925031e-16\\
5.7979797979798	2.22044604925031e-16\\
5.84848484848485	2.22044604925031e-16\\
5.8989898989899	2.22044604925031e-16\\
5.94949494949495	2.22044604925031e-16\\
6	2.22044604925031e-16\\
};
\addlegendentry{$\epsilon$}

\end{axis}
\end{tikzpicture}%
}
	
	\subfloat[$p=5$: fifth root\label{fig:preserving-pi-5}]{\definecolor{mycolor1}{rgb}{0.00000,0.44700,0.74100}%
\definecolor{mycolor2}{rgb}{0.85000,0.32500,0.09800}%
\begin{tikzpicture}

\begin{axis}[%
width=0.275\columnwidth,
height=1.2in,
at={(0\columnwidth,0in)},
scale only axis,
xmin=1,
xmax=6,
xtick={1,2,3,4,5,6},
xticklabels={{Unif.},{pth power of unif.},{exp of intensity},{K80 (Emb.)},{K80 (Not. Emb.)},{Pei}},
xticklabel style={rotate=90,font=\scriptsize},
ymode=log,
ymin=0.000164020760082955,
ymax=0.589082292065185,
yminorticks=true,
ylabel style={font=\color{white!15!black}},
ylabel={$\| X^5 - A \|_F$},
axis background/.style={fill=white},
legend style={at={(1.03,0.5)}, anchor=west, legend cell align=left, align=left, draw=white!15!black}
]
\addplot [color=mycolor1, only marks, mark size=3.5pt, mark=x, mark options={solid, mycolor1}]
  table[row sep=crcr]{%
1	0.589082292065185\\
2	0.00035266809933729\\
3	0.000164020760082955\\
4	0.422080394676951\\
5	0.000295696229626461\\
6	0.217456711137822\\
};

\addplot [color=mycolor2, only marks, mark size=3.5pt, mark=o, mark options={solid, mycolor2}]
  table[row sep=crcr]{%
1	0.439554495065458\\
2	0.000314605434291776\\
3	0.00022733958817614\\
4	0.0567759963838961\\
5	0.00201943566920878\\
6	0.125550416544272\\
};

\end{axis}

\begin{axis}[%
width=0.275\columnwidth,
height=1.2in,
at={(0.483\columnwidth,0in)},
scale only axis,
xmin=1,
xmax=6,
xtick={1,2,3,4,5,6},
xticklabels={{Unif.},{pth power of unif.},{exp of intensity},{K80 (Emb.)},{K80 (Not. Emb.)},{Pei}},
xticklabel style={rotate=90,font=\scriptsize},
ymode=log,
ymin=2.77555756156289e-17,
ymax=0.228992235264016,
yminorticks=true,
ylabel style={font=\color{white!15!black}, align=center},
ylabel={Infinity norm error between\\[1ex]target steady state and obtained one},
axis background/.style={fill=white},
legend style={at={(1.03,0.5)}, anchor=west, legend cell align=left, align=left, draw=none, fill=none}
]
\addplot [color=mycolor1, only marks, mark size=3.5pt, mark=x, mark options={solid, mycolor1}]
  table[row sep=crcr]{%
1	2.77555756156289e-17\\
2	5.55111512312578e-17\\
3	9.71445146547012e-17\\
4	5.55111512312578e-17\\
5	1.9595436384634e-14\\
6	4.85722573273506e-17\\
};
\addlegendentry{$S_n^\pi$}

\addplot [color=mycolor2, only marks, mark size=3.5pt, mark=o, mark options={solid, mycolor2}]
  table[row sep=crcr]{%
1	0.0160015955314558\\
2	0.0187056357298709\\
3	0.0175541883521085\\
4	0.228992235264016\\
5	0.00103342867672518\\
6	2.18268381418987e-05\\
};
\addlegendentry{$S_n$}

\addplot [color=black, dashed]
  table[row sep=crcr]{%
1	2.22044604925031e-16\\
1.05050505050505	2.22044604925031e-16\\
1.1010101010101	2.22044604925031e-16\\
1.15151515151515	2.22044604925031e-16\\
1.2020202020202	2.22044604925031e-16\\
1.25252525252525	2.22044604925031e-16\\
1.3030303030303	2.22044604925031e-16\\
1.35353535353535	2.22044604925031e-16\\
1.4040404040404	2.22044604925031e-16\\
1.45454545454545	2.22044604925031e-16\\
1.50505050505051	2.22044604925031e-16\\
1.55555555555556	2.22044604925031e-16\\
1.60606060606061	2.22044604925031e-16\\
1.65656565656566	2.22044604925031e-16\\
1.70707070707071	2.22044604925031e-16\\
1.75757575757576	2.22044604925031e-16\\
1.80808080808081	2.22044604925031e-16\\
1.85858585858586	2.22044604925031e-16\\
1.90909090909091	2.22044604925031e-16\\
1.95959595959596	2.22044604925031e-16\\
2.01010101010101	2.22044604925031e-16\\
2.06060606060606	2.22044604925031e-16\\
2.11111111111111	2.22044604925031e-16\\
2.16161616161616	2.22044604925031e-16\\
2.21212121212121	2.22044604925031e-16\\
2.26262626262626	2.22044604925031e-16\\
2.31313131313131	2.22044604925031e-16\\
2.36363636363636	2.22044604925031e-16\\
2.41414141414141	2.22044604925031e-16\\
2.46464646464646	2.22044604925031e-16\\
2.51515151515152	2.22044604925031e-16\\
2.56565656565657	2.22044604925031e-16\\
2.61616161616162	2.22044604925031e-16\\
2.66666666666667	2.22044604925031e-16\\
2.71717171717172	2.22044604925031e-16\\
2.76767676767677	2.22044604925031e-16\\
2.81818181818182	2.22044604925031e-16\\
2.86868686868687	2.22044604925031e-16\\
2.91919191919192	2.22044604925031e-16\\
2.96969696969697	2.22044604925031e-16\\
3.02020202020202	2.22044604925031e-16\\
3.07070707070707	2.22044604925031e-16\\
3.12121212121212	2.22044604925031e-16\\
3.17171717171717	2.22044604925031e-16\\
3.22222222222222	2.22044604925031e-16\\
3.27272727272727	2.22044604925031e-16\\
3.32323232323232	2.22044604925031e-16\\
3.37373737373737	2.22044604925031e-16\\
3.42424242424242	2.22044604925031e-16\\
3.47474747474747	2.22044604925031e-16\\
3.52525252525253	2.22044604925031e-16\\
3.57575757575758	2.22044604925031e-16\\
3.62626262626263	2.22044604925031e-16\\
3.67676767676768	2.22044604925031e-16\\
3.72727272727273	2.22044604925031e-16\\
3.77777777777778	2.22044604925031e-16\\
3.82828282828283	2.22044604925031e-16\\
3.87878787878788	2.22044604925031e-16\\
3.92929292929293	2.22044604925031e-16\\
3.97979797979798	2.22044604925031e-16\\
4.03030303030303	2.22044604925031e-16\\
4.08080808080808	2.22044604925031e-16\\
4.13131313131313	2.22044604925031e-16\\
4.18181818181818	2.22044604925031e-16\\
4.23232323232323	2.22044604925031e-16\\
4.28282828282828	2.22044604925031e-16\\
4.33333333333333	2.22044604925031e-16\\
4.38383838383838	2.22044604925031e-16\\
4.43434343434343	2.22044604925031e-16\\
4.48484848484848	2.22044604925031e-16\\
4.53535353535354	2.22044604925031e-16\\
4.58585858585859	2.22044604925031e-16\\
4.63636363636364	2.22044604925031e-16\\
4.68686868686869	2.22044604925031e-16\\
4.73737373737374	2.22044604925031e-16\\
4.78787878787879	2.22044604925031e-16\\
4.83838383838384	2.22044604925031e-16\\
4.88888888888889	2.22044604925031e-16\\
4.93939393939394	2.22044604925031e-16\\
4.98989898989899	2.22044604925031e-16\\
5.04040404040404	2.22044604925031e-16\\
5.09090909090909	2.22044604925031e-16\\
5.14141414141414	2.22044604925031e-16\\
5.19191919191919	2.22044604925031e-16\\
5.24242424242424	2.22044604925031e-16\\
5.29292929292929	2.22044604925031e-16\\
5.34343434343434	2.22044604925031e-16\\
5.39393939393939	2.22044604925031e-16\\
5.44444444444444	2.22044604925031e-16\\
5.49494949494949	2.22044604925031e-16\\
5.54545454545455	2.22044604925031e-16\\
5.5959595959596	2.22044604925031e-16\\
5.64646464646465	2.22044604925031e-16\\
5.6969696969697	2.22044604925031e-16\\
5.74747474747475	2.22044604925031e-16\\
5.7979797979798	2.22044604925031e-16\\
5.84848484848485	2.22044604925031e-16\\
5.8989898989899	2.22044604925031e-16\\
5.94949494949495	2.22044604925031e-16\\
6	2.22044604925031e-16\\
};
\addlegendentry{$\epsilon$}

\end{axis}
\end{tikzpicture}%
}
	
	\caption{On the left panel we report the error in Frobenius norm $\|X^p - A\|_F$ obtained by applying the \texttt{trustregions} method with initial guess \mintinline{matlab}{X0 = M.rand()} and tolerance of the gradient of \mintinline{matlab}{1e-4} for a matrix generated from the classes in Table~\ref{tab:matrix-classes}. The right panel reports the infinity norm error between $\boldsymbol{\pi}^T A = \boldsymbol{\pi}^T$ and $\tilde{\boldsymbol{\pi}}^T A = \tilde{\boldsymbol{\pi}}^T$, i.e., $\|\tilde{\boldsymbol{\pi}} - \boldsymbol{\pi}\|_\infty$. The dashed line represents the floating-point relative accuracy of MATLAB's double-precision number.}
\end{figure}
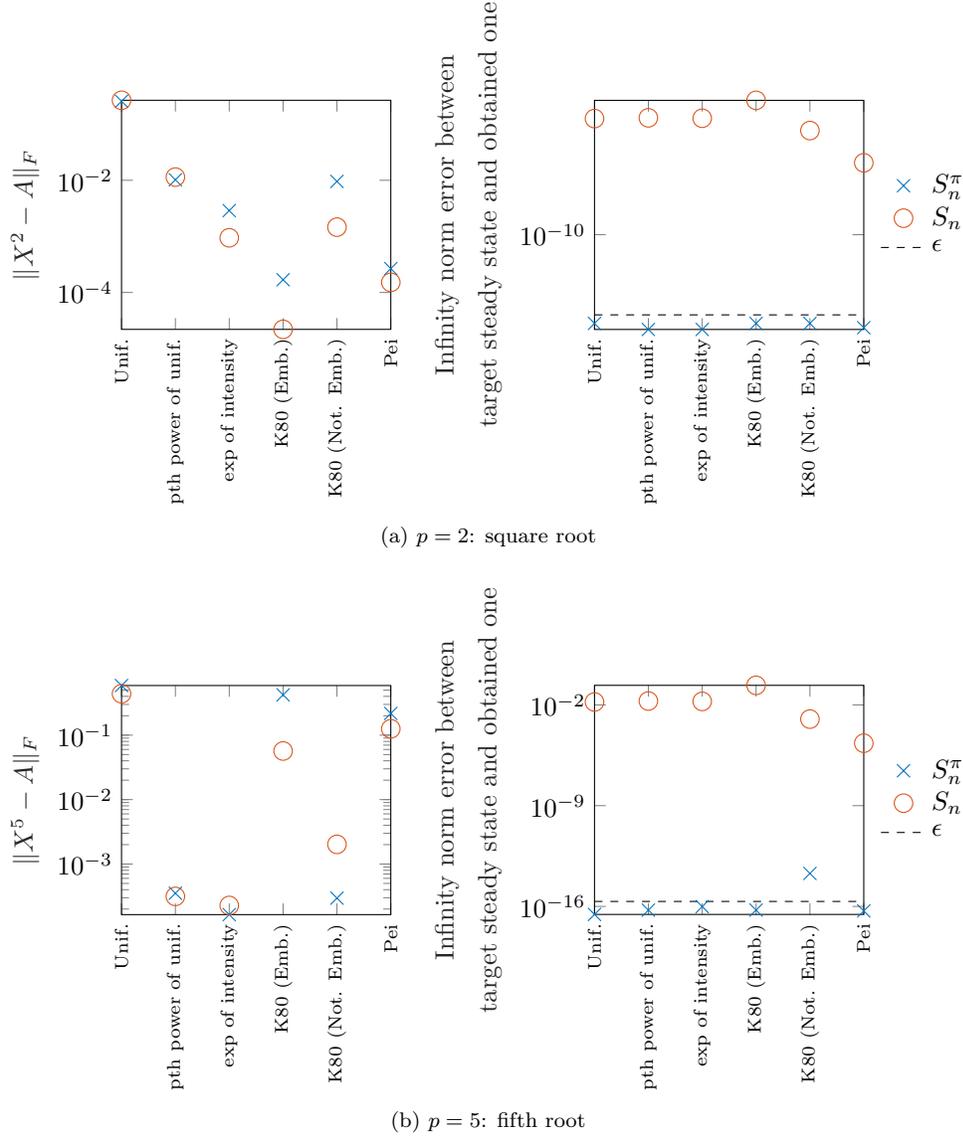

From Figure~\ref{fig:preserving-pi} we observe that in several cases the optimization routine on $\mathbb{S}_n^{\boldsymbol{\pi}}$ provides a larger or comparable residual for the approximation of the square root, furthermore in all the cases the stationary distribution is recovered to the floating-point relative accuracy. From this experiment, we observe a trade-off between the achievable accuracy with respect to the residual with the starting matrix and the recovery of the stationary vector. Analogous results are also obtained for the case in which we consider roots of higher order (Figure~\ref{fig:preserving-pi-5}) for which the behavior is essentially the same.

In the next sections, we investigate the behavior of the different numerical methods for the solution of the linear systems needed for computing the various projections between the tangent plane and the manifold discussed in Section~\ref{sec:computational_issues}.

\subsubsection{Properties and solution of the associated linear systems}\label{sec:experiment-on-bounds}

To validate the results of the bound given by the Proposition~\ref{pro:spectral_properties} we consider the test matrices from Table~\ref{tab:matrix-classes}. Specifically, we first build the manifold associated with the stationary distribution for each matrix in the 
class, then we generate a random point on that manifold and the associated $2 \times 2$ block matrix needed for computing the projection on the tangent space associated with that point, i.e., the matrix in~\eqref{eq:alpha_and_beta_values}. Figure~\ref{fig:bound-on-matrices} reports the result of such an experiment.
\begin{figure}[htbp]
	\centering
	\input{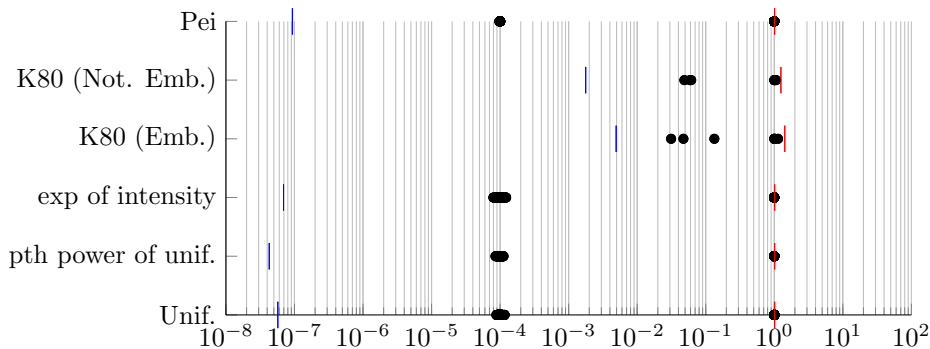}
	
	\caption{Bound from Proposition~\ref{pro:spectral_properties} for the matrices built on the $\mathbb{S}_n^{\boldsymbol{\pi}}$ manifolds associated with the stationary distribution of the test matrices from Table~\ref{tab:matrix-classes}. The vertical bars denote the bounds (red for the upper, blue for the lower), and the black dots are the numerically computed eigenvalues.}
	\label{fig:bound-on-matrices}
\end{figure}
We observe that the upper bound is more accurate than the lower one.

As a second exploration, we consider the different solution strategies of associated linear systems discussed in Section~\ref{sec:computational_issues}. Let us first consider a case outside the optimization algorithm. That is, we generate the manifold $\mathbb{S}_n^{\boldsymbol{\pi}}$, where $\boldsymbol{\pi}$ is the stationary distribution associated with the out-degree random walk on the graphs from Figure~\ref{tab:graph-example}.
\begin{figure}[htb]
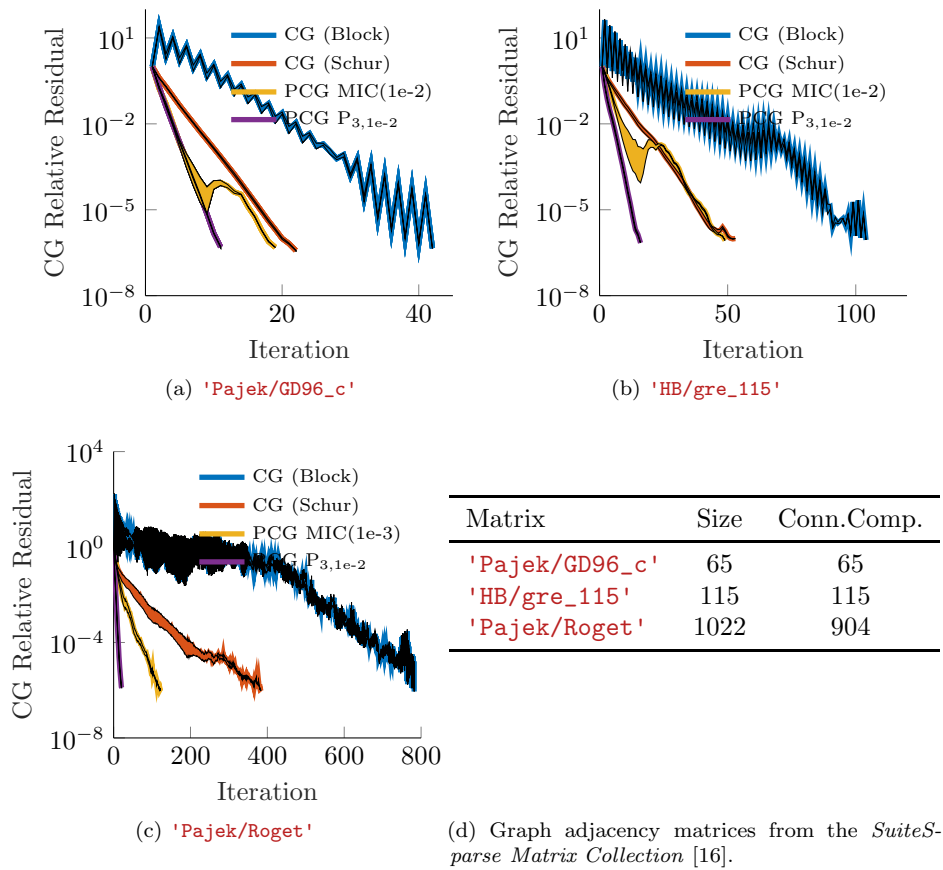

	\centering
	\subfloat[\mintinline{matlab}{'Pajek/GD96_c'}]{\input{Pajek_GD96_c}}
	\subfloat[\mintinline{matlab}{'HB/gre_115'}]{\input{HB_gre_115}}
	
	\subfloat[\mintinline{matlab}{'Pajek/Roget'}]{\input{Pajek_Roget}}
	\subfloat[\label{tab:graph-example}Graph adjacency matrices from the \emph{SuiteSparse Matrix Collection}~\cite{SuiteSparse}.]{%
		\raisebox{1.2in}{\begin{tabular}{p{0.20\textwidth}cc}
				\toprule
				Matrix  & Size & Conn.Comp.  \\
				\midrule    
				\mintinline{matlab}{'Pajek/GD96_c'} & 65 & 65 \\
				\mintinline{matlab}{'HB/gre_115'} & 115 & 115 \\
				\mintinline{matlab}{'Pajek/Roget'} & 1022 & 904 \\ 
				\bottomrule
		\end{tabular}}%
	}
	\caption{Reduction of the relative residual on the average of 50 external iterations of the optimization algorithm. The solid line represents the mean, the shaded area of the same color the 95\% confidence interval. In Figure~\ref{tab:graph-example} we give the size of the original graph and of its largest strongly connected component.}
	\label{fig:iteration}
\end{figure}
Specifically, if we call $G$ the adjacency matrix of the largest connected component of such graphs, then we consider the manifold $\mathbb{S}_n^{\boldsymbol{\pi}}$ for $\boldsymbol{\pi}$ the invariant vector of the stochastic matrix $A = \operatorname{diag}(G\mathbf{1})^{-1}\mathbf{G}$. To have comparable results in the various cases, we choose to always have the external Riemannian trust-region optimizer perform 50 iterations and all instances are run by re-initializing the random number generator to the same seed. From the results in Figure~\ref{fig:iteration} we observe that the $2 \times 2$-block formulation is the one for which the convergence is more prone to undergo oscillations. On the other hand, the reduced reformulation in terms of the Schur complement given in equation~\eqref{eq:schurversion} tends to improve the situation. It is also observed that the two preconditioning strategies proposed are able to consistently reduce the number of iterations necessary for convergence. In particular, the version that uses the modified Cholesky incomplete factorization is more sensitive to the choice of the tolerance on the discarded elements and on the size of the global matrix. On the other hand, the use of the Neumann series preconditioner $P_{k,\tau}$ seems to give more consistent results with a higher tolerance on the elements to be dropped.

\subsection{The case of reducible Markov chains}\label{sec:reducible-chains}

We consider here what is an \emph{edge case} for our approach, i.e.,  when the chain is reducible due to the existence of two communication classes; in other words, to the case in which the stationary distribution $\boldsymbol{\pi} = [0,\ldots,0,1]^T$. This case is motivated by the embedding problem for Markov models of the term structure of credit risk spreads~\cite{application2}. The Markov chain modeling represents the dynamics of the different credit rating states, i.e., evaluations of the relative ability of an entity or obligation to meet financial commitments over time and examining the probability of transitioning in between these states. The different levels go from \texttt{AAA}, the lowest expectation of default risk, to \texttt{D}, for an issuer who has entered into bankruptcy and cannot recover from it. In Table~\ref{tab:creditriskmatrix} we report sample data from~\cite{application2}.
\begin{table}[htbp]
	\centering
	\small
	\begin{tabular}{r|cccccccc}
		\toprule
		\mintinline{matlab}{data(i,j)} &	\texttt{AAA}	&	\texttt{AA}	&	\texttt{A}	&	\texttt{BBB}	&	\texttt{BB}	&	\texttt{B}	&	\texttt{CCC}	&	\texttt{D}	\\
		\midrule
		\texttt{AAA}	&	0,891	&	0,0963	&	0,0078	&	0,0019	&	0,003	&	0	&	0	&	0	\\
		\texttt{AA}	&	0,0086	&	0,901	&	0,0747	&	0,0099	&	0,0029	&	0,0029	&	0	&	0	\\
		\texttt{A}	&	0,0009	&	0,0291	&	0,8894	&	0,0649	&	0,0101	&	0,0045	&	0	&	0,0009	\\
		\texttt{BBB}	&	0,0006	&	0,0043	&	0,0656	&	0,8427	&	0,0644	&	0,016	&	0,0018	&	0,0045	\\
		\texttt{BB}	&	0,0004	&	0,0022	&	0,0079	&	0,0719	&	0,7764	&	0,1043	&	0,0127	&	0,0241	\\
		\texttt{B}	&	0	&	0,0019	&	0,0031	&	0,0066	&	0,0517	&	0,8246	&	0,0435	&	0,0685	\\
		\texttt{CCC}	&	0	&	0	&	0,0116	&	0,0116	&	0,0203	&	0,0754	&	0,6493	&	0,2319	\\
		\texttt{D}	&	0	&	0	&	0	&	0	&	0	&	0	&	0	&	1	\\
		\bottomrule
	\end{tabular}
	\cprotect\caption{Standard and Poor's Credit Review (1993) from \cite[Table~3]{application2}. The numbers are reported with four figures as in the original data, this means that the probabilities are not normalized to have sum one. To be able to use them in the code we renormalize the entries as: {\mintinline{matlab}{A = diag(sum(data,2))\data;}}.}
	\label{tab:creditriskmatrix}
\end{table}
In order to compute a stochastic approximation of the root of the matrix, we cannot use manifold-based optimization directly, since the stationary distribution is $\boldsymbol{\pi} = [0,\ldots,0,1]^T$ and therefore it cannot be used to construct the manifold. In a similar guise to the Page Rank problem, we  apply a perturbation to the data matrix to make it irreducible, as
\[
\tilde{A} = (1-\gamma) A +  \gamma(\mathbf{1}\mathbf{1}^T)/n, \qquad 0 < \gamma \ll 1,  
\]
which admits stationary distribution $\tilde{\boldsymbol{\pi}}>0$ with respect to which we can construct the manifold $\mathbb{S}_n^{\tilde{\boldsymbol{\pi}}}$. The approximate root can then be computed by solving for
\begin{equation}\label{eq:riskopt}
	\min_{X \in \mathbb{S}_n^{\tilde{\boldsymbol{\pi}}}} \frac{1}{2}\|X^2 - A \|_F^2.
\end{equation}
For $\gamma = 10^{-4}$, we obtain \[\tilde{\boldsymbol{\pi}} \approxeq [0.0002, 0.0007, 0.0012, 0.0009, 0.0005, 0.0006, 0.0001, 0.9957]^T. \]
For the construction of the starting point for the optimization we exploit the modified Sinkhorn algorithm (Proposition~\ref{prop:extsk}) to project the perturbed version $A$ of the upper triangular matrix on the manifold in the same way,
\begin{minted}[bgcolor=bg,fontsize=\small]{matlab}
E = ones(size(A));                %
X0 = diag(sum(triu(ones(size(A))),2))\triu(ones(size(A))); %
X0 = gam*E + (1-gam)*X0;          %
X0 = modifiedsinkhorn(X0,pi,100); %
\end{minted}
and start both the standard constrained optimization and the Riemannian optimization for~\eqref{eq:riskopt}. In Figure~\ref{fig:riskfigure} we report the two approximations of the root thus obtained, as it can be observed the approximation obtained through the Riemannian optimization has a structure that is much closer to what we would expect from the appropriate matrix function. In particular the extradiagonal values on the last row are negligible with respect to the weight of the transition probability of remaining in \texttt{D}.
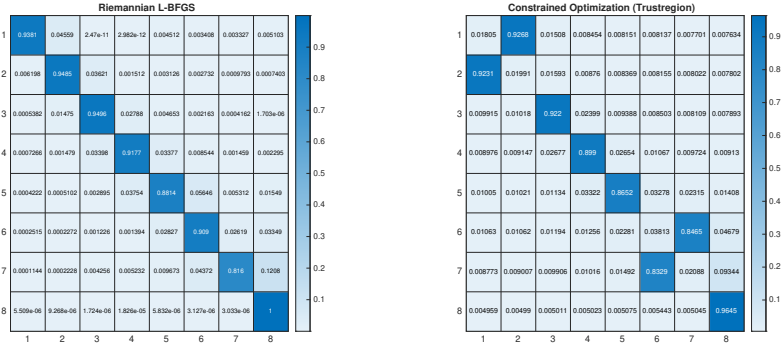
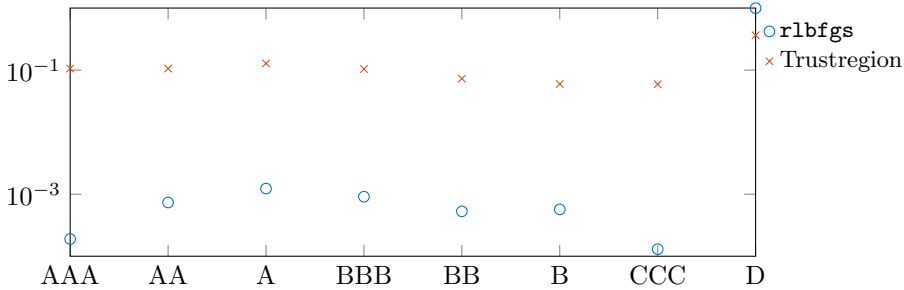
\begin{figure}
	\centering%
	\definecolor{mycolor1}{rgb}{0.00000,0.44700,0.74100}%
	\definecolor{mycolor2}{rgb}{0.85000,0.32500,0.09800}%
	\subfloat[Entries of the approximated square roots]{\includegraphics[width=0.8\columnwidth]{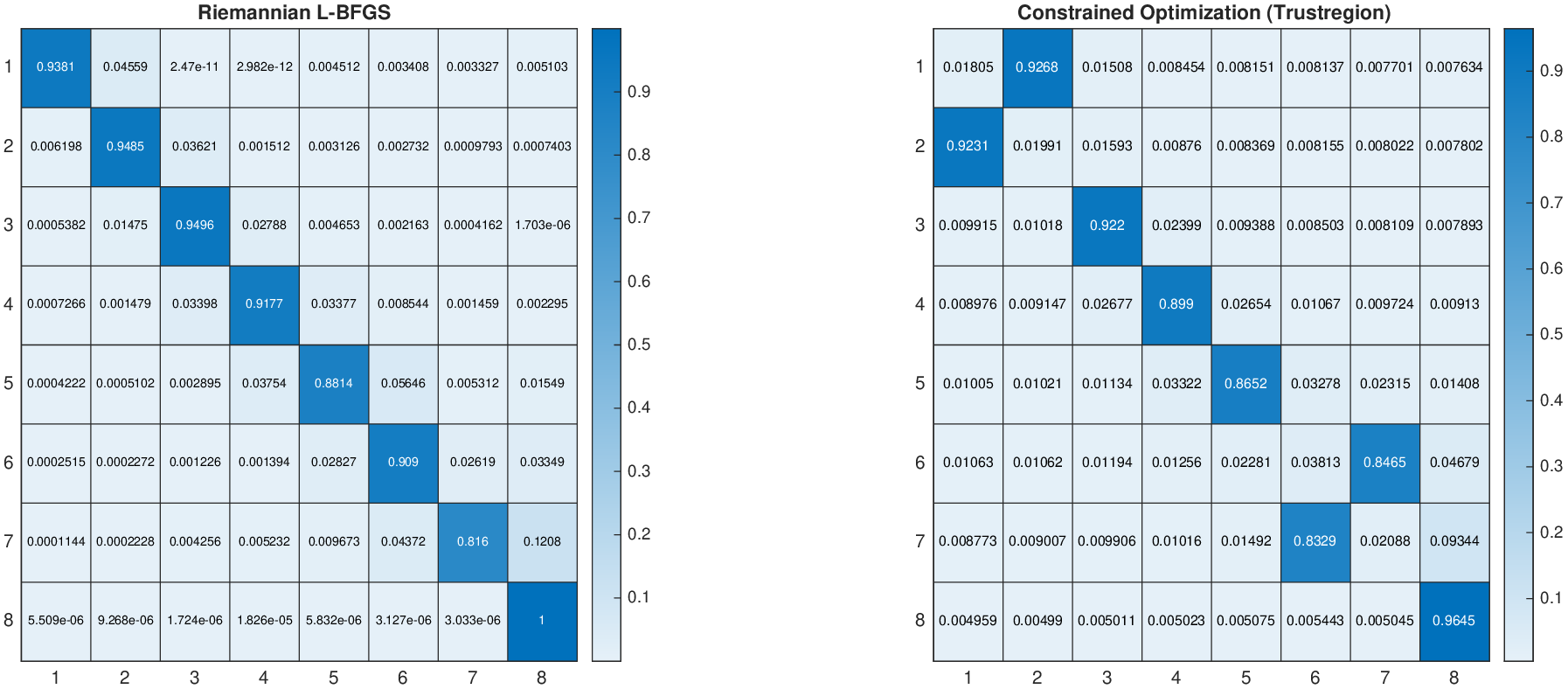}}
	
	\subfloat[stationary distributions\label{fig:riskfigure:stationary}]{%
		\begin{tikzpicture}
			
			\begin{axis}[%
				width=0.7\columnwidth,
				height=1.3in,
				at={(0in,0in)},
				scale only axis,
				xmin=1,
				xmax=8,
				xtick={1,2,3,4,5,6,7,8},
				xticklabels={{AAA},{AA},{A},{BBB},{BB},{B},{CCC},{D}},
				ymode=log,
				ymin=0.0001,
				ymax=1,
				yminorticks=true,
				axis background/.style={fill=white},
				legend style={at={(0.7\columnwidth,0.7)}, anchor=south west, legend cell align=left, align=left, draw=none, fill=none, font=\small}
				]
				\addplot [color=mycolor1, only marks, mark=o, mark options={solid, mycolor1}]
				table[row sep=crcr]{%
					1	0.000189568645599814\\
					2	0.000734159310737329\\
					3	0.00123114546198187\\
					4	0.000911392242344639\\
					5	0.000529392887304481\\
					6	0.000568798528368104\\
					7	0.000129996156475851\\
					8	0.995705546767188\\
				};
				\addlegendentry{\texttt{rlbfgs}}
				
				\addplot [color=mycolor2, only marks, mark=x, mark options={solid, mycolor2}]
				table[row sep=crcr]{%
					1	0.105732841609923\\
					2	0.106090607115502\\
					3	0.128482109653194\\
					4	0.104026681119782\\
					5	0.0727518860646815\\
					6	0.0597796907785042\\
					7	0.059077910901835\\
					8	0.364058272756578\\
				};
				\addlegendentry{Trustregion}
				
			\end{axis}
		\end{tikzpicture}%
	}
	\caption{Approximate square roots via the Riemannian L-BFGS algorithm (\texttt{rlbfgs}) and the constrained optimization algorithm started from the same approximation $X_0$ for the Markov models of the term structure of credit risk spreads.}
	\label{fig:riskfigure}
\end{figure}
In addition, beyond the structural similarity, the Riemannian approximation preserves the stationary (perturbed) distribution by construction, while the one based solely on constrained optimization produces a probability distribution that is closer to being uniform than concentrated in the \texttt{D} state; see Figure~\ref{fig:riskfigure:stationary}.

\section{Conclusions and future directions}\label{sec:conclusions}

In this paper, we have dealt with the problem of approximating the $p$th root of a stochastic matrix with a stochastic matrix. In particular, by observing that stochastic matrices form a Riemannian manifold with respect to the Fisher metric, we have exploited several specific optimization algorithms which show better performance than their counterparts that use only constrained optimization. Furthermore, we have introduced a new Riemannian manifold---employing the same metric---on the set of stochastic matrices with fixed steady state. This allowed us to employ Riemannian optimization algorithms capable of obtaining approximations of stochastic $p$th roots which also preserve such vector, i.e., such that the Markov chain induced by them has the same steady state. We have also shown that we can apply the proposed strategy to the case of non-reducible Markov chains through a perturbation technique. 

In the future, we intend to study the geodetic structure of this new manifold to better characterize the obtained $p$th root approximations, and to further investigate the solution of the associated computational problems, e.g., the solution of the linear systems needed to calculate the projections from the tangent plane to the manifold and the representation of the Riemannian Hessian, in order to further improve the computational efficiency of the proposed methods.

\section*{Acknowledgements} 
We thank the two anonymous reviewers whose comments allowed us to improve the presentation and clarity of the paper.

\bibliographystyle{siamplain}
\bibliography{references}

\begin{thebibliography}{10}

\bibitem{genrtr}
{\sc P.-A. Absil, C.~G. Baker, and K.~A. Gallivan}, {\em Trust-region methods
  on {Riemannian} manifolds}, Foundations of Computational Mathematics, 7
  (2007), pp.~303--330, \url{https://doi.org/10.1007/s10208-005-0179-9}.

\bibitem{AbsilBook}
{\sc P.-A. Absil, R.~Mahony, and R.~Sepulchre}, {\em Optimization algorithms on
  matrix manifolds}, Princeton University Press, Princeton, NJ, 2008,
  \url{https://doi.org/10.1515/9781400830244}.
\newblock With a foreword by Paul Van Dooren.

\bibitem{Ardiyansyah2021}
{\sc M.~Ardiyansyah, D.~Kosta, and K.~Kubjas}, {\em The model-specific markov
  embedding problem for symmetric group-based models}, Journal of Mathematical
  Biology, 83 (2021), \url{https://doi.org/10.1007/s00285-021-01656-5}.

\bibitem{Baake2020}
{\sc M.~Baake and J.~Sumner}, {\em Notes on markov embedding}, Linear Algebra
  and Its Applications, 594 (2020), p.~262 – 299,
  \url{https://doi.org/10.1016/j.laa.2020.02.016}.

\bibitem{application1}
{\sc J.~Beck and S.~G. Pauker}, {\em {The Markov Process in Medical
  Prognosis}}, Medical Decision Making, 3 (1983), pp.~419--458,
  \url{https://doi.org/10.1177/0272989X8300300403}.
\newblock PMID: 6668990.

\bibitem{BenziFaccio}
{\sc M.~Benzi and C.~Faccio}, {\em {Solving Linear Systems of the Form
  {\({\boldsymbol{(A + \gamma UU^T)\, {x} = {b}}}\)} by Preconditioned
  Iterative Methods}}, SIAM Journal on Scientific Computing, 0 (0),
  pp.~S51--S70, \url{https://doi.org/10.1137/22M1505529}.

\bibitem{bp:book}
{\sc A.~Berman and R.~J. Plemmons}, {\em Nonnegative matrices in the
  mathematical sciences}, vol.~9 of Classics in Applied Mathematics, Society
  for Industrial and Applied Mathematics (SIAM), Philadelphia, PA, 1994,
  \url{https://doi.org/10.1137/1.9781611971262}.
\newblock Revised reprint of the 1979 original.

\bibitem{Bhat2020}
{\sc B.~R. Bhat, R.~Hillier, N.~Mallick, and V.~K. U.}, {\em Roots of
  completely positive maps}, Linear Algebra and Its Applications, 587 (2020),
  p.~143 – 165, \url{https://doi.org/10.1016/j.laa.2019.10.027}.

\bibitem{blm:book}
{\sc D.~A. Bini, G.~Latouche, and B.~Meini}, {\em Numerical methods for
  structured {M}arkov chains}, Numerical Mathematics and Scientific
  Computation, Oxford University Press, New York, 2005,
  \url{https://doi.org/10.1093/acprof:oso/9780198527688.001.0001}.
\newblock Oxford Science Publications.

\bibitem{MR4533407}
{\sc N.~Boumal}, {\em An introduction to optimization on smooth manifolds},
  Cambridge University Press, Cambridge, 2023.

\bibitem{manopt}
{\sc N.~Boumal, B.~Mishra, P.-A. Absil, and R.~Sepulchre}, {\em {M}anopt, a
  {M}atlab {T}oolbox for {O}ptimization on {M}anifolds}, Journal of Machine
  Learning Research, 15 (2014), pp.~1455--1459, \url{https://www.manopt.org}.

\bibitem{interiorpoint}
{\sc R.~H. Byrd, M.~E. Hribar, and J.~Nocedal}, {\em An interior point
  algorithm for large-scale nonlinear programming}, vol.~9, 1999, pp.~877--900,
  \url{https://doi.org/10.1137/S1052623497325107}.
\newblock Dedicated to John E. Dennis, Jr., on his 60th birthday.

\bibitem{Casanellas2020}
{\sc M.~Casanellas, J.~Fernández-Sánchez, and J.~Roca-Lacostena}, {\em
  Embeddability and rate identifiability of kimura 2-parameter matrices},
  Journal of Mathematical Biology, 80 (2020), p.~995 – 1019,
  \url{https://doi.org/10.1007/s00285-019-01446-0}.

\bibitem{https://doi.org/10.1002/sim.2970}
{\sc T.~Charitos, P.~R. de~Waal, and L.~C. van~der Gaag}, {\em Computing
  short-interval transition matrices of a discrete-time markov chain from
  partially observed data}, Statistics in Medicine, 27 (2008), pp.~905--921,
  \url{https://doi.org/https://doi.org/10.1002/sim.2970}.

\bibitem{Davies2010}
{\sc E.~Davies}, {\em Embeddable {M}arkov matrices}, Electronic Journal of
  Probability, 15 (2010), p.~1474 – 1486,
  \url{https://doi.org/10.1214/EJP.v15-733}.

\bibitem{SuiteSparse}
{\sc T.~A. Davis and Y.~Hu}, {\em {T}he {U}niversity of {F}lorida {S}parse
  {M}atrix {C}ollection}, ACM Trans. Math. Softw., 38 (2011),
  \url{https://doi.org/10.1145/2049662.2049663}.

\bibitem{PerformanceProfiles}
{\sc E.~D. Dolan and J.~J. Mor\'{e}}, {\em Benchmarking optimization software
  with performance profiles}, Math. Program., 91 (2002), pp.~201--213,
  \url{https://doi.org/10.1007/s101070100263}.

\bibitem{Douik8861409}
{\sc A.~Douik and B.~Hassibi}, {\em Manifold optimization over the set of
  doubly stochastic matrices: A second-order geometry}, IEEE Transactions on
  Signal Processing, 67 (2019), pp.~5761--5774,
  \url{https://doi.org/10.1109/TSP.2019.2946024}.

\bibitem{Ekhosuehi2023}
{\sc V.~U. Ekhosuehi}, {\em On the use of cauchy integral formula for the
  embedding problem of discrete-time markov chains}, Communications in
  Statistics - Theory and Methods, 52 (2023), p.~973 – 987,
  \url{https://doi.org/10.1080/03610926.2021.1921806}.

\bibitem{application3}
{\sc K.~Gabriel and J.~Neumann}, {\em A {M}arkov chain model for daily rainfall
  occurrence at {T}el {A}viv}, Quarterly Journal of the Royal Meteorological
  Society, 88 (1962), p.~90 – 95,
  \url{https://doi.org/10.1002/qj.49708837511}.

\bibitem{HighamBook}
{\sc N.~J. Higham}, {\em Functions of matrices}, Society for Industrial and
  Applied Mathematics (SIAM), Philadelphia, PA, 2008,
  \url{https://doi.org/10.1137/1.9780898717778}.
\newblock Theory and computation.

\bibitem{HighamLijing2010}
{\sc N.~J. Higham and L.~Lin}, {\em On {$p$}th roots of stochastic matrices},
  Linear Algebra Appl., 435 (2011), pp.~448--463,
  \url{https://doi.org/10.1016/j.laa.2010.04.007}.

\bibitem{MR2978290}
{\sc R.~A. Horn and C.~R. Johnson}, {\em Matrix analysis}, Cambridge University
  Press, Cambridge, second~ed., 2013.

\bibitem{Huang2016}
{\sc W.~Huang, P.-A. Absil, and K.~Gallivan}, {\em A {R}iemannian {BFGS} Method
  for Nonconvex Optimization Problems}, Springer International Publishing,
  Cham, 2016, pp.~627--634, \url{https://doi.org/10.1007/978-3-319-39929-4_60}.

\bibitem{Hughes2016}
{\sc M.~Hughes and R.~Werner}, {\em Choosing {M}arkovian credit migration
  matrices by nonlinear optimization}, Risks, 4 (2016),
  \url{https://doi.org/10.3390/risks4030031}.

\bibitem{application2}
{\sc R.~A. Jarrow, D.~Lando, and S.~M. Turnbull}, {\em {A Markov Model for the
  Term Structure of Credit Risk Spreads}}, The Review of Financial Studies, 10
  (2015), pp.~481--523, \url{https://doi.org/10.1093/rfs/10.2.481}.

\bibitem{Kingman}
{\sc J.~F.~C. Kingman}, {\em The imbedding problem for finite {M}arkov chains},
  Z. Wahrscheinlichkeitstheorie und Verw. Gebiete, 1 (1962), pp.~14--24,
  \url{https://doi.org/10.1007/BF00531768}.

\bibitem{LinThesis}
{\sc L.~Lin}, {\em {Roots of Stochastic Matrices and Fractional Matrix
  Powers}},
  \url{https://research.manchester.ac.uk/files/54504098/FULL_TEXT.PDF}.

\bibitem{Pfeuffer2017}
{\sc M.~Pfeuffer}, {\em Ctmcd: An r package for estimating the parameters of a
  continuous-time markov chain from discrete-time data}, R Journal, 9 (2017),
  p.~127 – 141, \url{https://doi.org/10.32614/rj-2017-038}.

\bibitem{Rothblum}
{\sc U.~G. Rothblum and H.~Schneider}, {\em Scalings of matrices which have
  prespecified row sums and column sums via optimization}, Linear Algebra and
  Its Applications, 114-115 (1989), p.~737 – 764,
  \url{https://doi.org/10.1016/0024-3795(89)90491-6}.

\bibitem{MR1990645}
{\sc Y.~Saad}, {\em Iterative methods for sparse linear systems}, Society for
  Industrial and Applied Mathematics, Philadelphia, PA, second~ed., 2003,
  \url{https://doi.org/10.1137/1.9780898718003}.

\bibitem{7182334}
{\sc Y.~Sun, J.~Gao, X.~Hong, B.~Mishra, and B.~Yin}, {\em Heterogeneous
  {T}ensor {D}ecomposition for {C}lustering via {M}anifold {O}ptimization},
  IEEE Transactions on Pattern Analysis and Machine Intelligence, 38 (2016),
  pp.~476--489, \url{https://doi.org/10.1109/TPAMI.2015.2465901}.

\bibitem{vanbrunt}
{\sc A.~Van-Brunt}, {\em Infinitely divisible nonnegative matrices,
  {$M$}-matrices, and the embedding problem for finite state stationary
  {M}arkov chains}, Linear Algebra Appl., 541 (2018), pp.~163--176,
  \url{https://doi.org/10.1016/j.laa.2017.11.018}.

\bibitem{Veerman2019}
{\sc J.~Veerman and E.~Kummel}, {\em Diffusion and consensus on weakly
  connected directed graphs}, Linear Algebra and Its Applications, 578 (2019),
  p.~184 – 206, \url{https://doi.org/10.1016/j.laa.2019.05.014}.

\end{thebibliography}
\end{document}